\documentclass{article}
\usepackage{amsfonts}
\usepackage{graphicx}
\usepackage{amssymb}
\usepackage{amsmath}
\usepackage{a4wide}
\usepackage{graphicx}
\usepackage{caption}
\usepackage{subcaption}
\usepackage{setspace}
\usepackage[affil-it]{authblk}
\usepackage{xr}
\usepackage{color}
\usepackage{xcolor,soul}
\usepackage{mathtools}
\usepackage{setspace}
\usepackage{float}

\usepackage[utf8]{inputenc} 
\usepackage[T1]{fontenc}

\newtheorem{theorem}{Theorem}[section]

\newtheorem{corollary}[theorem]{Corollary}

\newtheorem{proposition}[theorem]{Proposition}
\newtheorem{remark}[theorem]{Remark}

\newenvironment{proof}[1][Proof]{\textbf{#1.} }{\ \rule{0.5em}{0.5em}}
\newcommand{\mathsym}[1]{{}}
\newcommand{\unicode}[1]{{}}

\begin{document}

\title{Generalized Zermelo navigation on Hermitian manifolds under mild wind}

\author{Nicoleta Aldea$^{3}$, Piotr Kopacz$^{1,2}$}

\affil{\small{$^1$Jagiellonian University, Faculty of Mathematics and Computer Science\\ 6, Prof. St. {\L}ojasiewicza, 30 - 348, Krak\'{o}w, Poland}}
\affil{$^2$Gdynia Maritime University, Faculty of Navigation \\3,  Al. Jana Paw{\l}a II, 81-345, Gdynia, Poland}
\affil{\small{$^3$Transilvania University, Faculty of Mathematics and Informatics\\ 50, Iuliu Maniu, Bra\c{s}ov, Romania}}

\date{}
\date{{\normalsize {\small {e-mail:} \texttt{nicoleta.aldea@lycos.com, piotr.kopacz@im.uj.edu.pl,}}}}
\maketitle

\begin{abstract}
\noindent We generalize and study the Zermelo navigation problem on Hermitian manifolds in the \linebreak  presence of a perturbation $W$ determined by a mild complex velocity vector field \linebreak $||W(z)||_h<||u(z)||_h$, with application of complex Finsler metric of complex Randers type. By admitting space-dependence of ship's relative speed $||u(z)||_h\leq1$ we discuss the projectively related complex Finsler metrics, the geodesics corresponding to the solutions of Zermelo's problem, in particular the conformal case and the connections between the corresponding background Hermitian metric $h$, new Hermitian metric $a$ and resulting complex Randers metric $F$. Moreover, we present some necessary and sufficient conditions for the obtained locally projectively flat solutions. Our findings are also illustrated with several examples.
\end{abstract}

\setstcolor{red} 

\bigskip \noindent \textbf{M.S.C. 2010}: 53B20, 53C21, 53C22, 53C60, 49J15, 49J53.



\section{Introduction}

\noindent In the navigation problem, formulated initially in low dimensional
Euclidean spaces and solved with original method in variational calculus by
Ernst Zermelo (1931) \cite{zermelo}, the objective is to find the minimum
time trajectory of a ship sailing on a sea $M$, with the presence of a wind
determined by a vector field $W$. We aim to consider the problem as purely
geometric.
The essential generalization of  the problem on
Riemannian manifolds $(M, h)$ has been presented in Finsler geometry  (2004)  \cite{colleen_shen} in the original formulation, that is a
ship proceeds at constant unit speed relative to a surrounding mild perturbation $W$ thought of a wind,
i.e., $h(W,W)<1$. It was pointed out that the solutions are the
geodesics of a real Randers space $(M,F)$, a special Finsler space,
determined by a Riemannian metric $h$ and a vector field $W$, and
conversely, every real Randers metric arises from such a problem. In the
absence of a perturbation the solutions to the problem are simply $h$%
-geodesics of $M$. The important approach in discussing Randers metrics is
navigation representation, namely express Randers metric in terms of a
background Riemannian metric and a vector field. The corresponding condition on strong convexity,
i.e. $|W|_{h}<1$ ensures then that $F$ is a positive definite Finsler
metric  \cite{colleen_shen, chern_shen}.

Having applied Zermelo navigation it was possible to set up a complete list up to
local isometry of strongly convex Randers metrics of constant (Finslerian)
flag curvature. In particular, the resulting Finsler metrics have been
categorized, depending on the sign of their flag curvature. Also, the
formulae of all infinitesimal homotheties $W$ of the three standard
Riemannian space forms were worked out. With the use of the navigation
technique the classification of strongly convex Randers metrics of constant
flag curvature, both local and global, was given. Moreover, this enabled to
discuss projectively flat Randers metrics of constant flag curvature and the
corresponding PDE, the geodesics of Randers spaces of constant curvature
\cite{cr} and, more general, the geodesics of Finsler metrics \cite{huang_mo}.

Thereafter, being motivatad by some real applications the problem was investigated by the second author on Riemannian manifolds, having admitted the space dependence of ship's relative speed, with the presence of both mild and critical wind $W$. Setting as a reference point Zermelo's formulation of the problem, we asked whether a ship had to proceed at a constant maximum speed relative to the surrounding Riemannian sea. This standard assumption we have already dropped in purely geometric approach (cf. \cite{kopi6, kopi7}). Hence, the generalization included the problem in the original setting with $|u|_h=1=const.$ as a particular case. Furthermore, the concept
with admitting variable in space ship's speed gives rise to optimize our
recent study of the search models based on time-minimal paths in
real navigational applications via the Zermelo navigation (cf. \cite{kopi4}).

The navigation problem was also considered in complex Finsler geometry with application of complex Randers metric by the first author \cite{nicoleta}%
. In contrast to a real analogue, a complex Randers metric was
introduced and commenced to be studied much later (2007, cf. \cite{AM}). In order to benefit from similar
interest like this given by real Randers metrics, it was natural to find
some applications for them. Thus, the concepts in real setting presented \cite{colleen_shen} were
referred in \cite{nicoleta}, where the navigation problem was investigated  on a Hermitian manifold. To obtain a solution as a complex Randers
metric, it was necessary to work out the additional geometric assumptions.
Application of navigation representation in a complex landscape enabled to
obtain the concrete examples of complex Randers metrics and to point out the
essential difference in comparison to the analogous problem on Riemannian
manifolds. Namely, the complex Randers metrics are not of constant
holomorphic curvature by perturbation of some Hermitian metrics of constant
holomorphic sectional curvature via the Zermelo navigation. However, many
other properties related to it and the corresponding generalizations can be
considered, for instance projectively flat complex Randers metrics of zero holomorphic curvature, the geodesics corresponding to
some projectively related solutions, a special case of orthogonality, the
complex $(\alpha, \beta)$-solutions which are not of Randers type.

In the current paper we extend and refine our previous findings in a complex original
formulation of the Zermelo navigation problem as well as we obtain some new results with  the generalized setting. In our discussion we combine both our studies of the problem, in particular considering varying in magnitude speed $%
||u(z)||_{h}$ on Hermitian manifolds $M$ for the case of a complex mild
perturbation $W$, that is $||W(z)||_{h}<||u(z)||_{h}\leq 1$ everywhere on $M$%
.

We proceed by first presenting an overview of our paper's content.
In Section 2 some preliminary notions on $n$-dimensional complex Finsler
spaces, in particular on complex Randers spaces and projectively related
complex Finsler metrics, are stated. Also, a necessary and sufficient
condition for a complex Finsler metric to be locally projectively flat
is proved (see Theorem 2.4).
In Section 3 we describe the generalized Zermelo navigation problem on
Hermitian manifolds. Comparing to its Riemannian analogue we show that the
obtained complex Finsler solution is conditioned by an additional
assumption. Moreover, the orthogonal case is pointed out (see Theorem 3.8).
Making use of the generalization presented in a complex setting, in Section 4 we discuss on
projectively related complex Finsler metrics, the geodesics corresponding to
the solutions of Zermelo's problem, in particular the conformal solutions
and the connections between background Hermitian metric $h$, new Hermitian
metric $a$ and a resulting complex Randers metric $F$ (see Theorems 4.9 and
4.15). Furthermore, some necessary and sufficient conditions for obtained locally
projectively flat solutions are given (see Theorem 4.13 and
Corollary 4.17). Our findings are also illustrated with several examples.


\section{Preliminaries}

To begin with, we point out only the basic notions from complex Finsler geometry and
then briefly recall complex Randers metrics, for more information see \cite%
{A-P, Mub, Al-Mu3, C-S, W-Z}. Also,  the main results on projectively related complex Finsler metrics and locally projectively flat metrics are mentioned.

\subsection{Basic notions and notations}

Let $M$ be an $n$-dimensional complex manifold and $z=(z^{k})_{k=\overline{%
1,n}}$ be the complex coordinates in a local chart. The complexified $T_{%
\mathbb{C}}M$ of the real tangent bundle $T_{\mathbb{R}}M$ splits into the
sum of the holomorphic tangent bundle $T^{\prime }M$ and its conjugate $%
T^{\prime \prime }M$. The bundle $T^{\prime }M,$ ($\pi :T^{\prime
}M\rightarrow M)$ is itself a complex manifold and the local coordinates in
a local chart will be denoted by $\mu =(z^{k},\eta ^{k})_{k=\overline{1,n}}.$
These are changed into $(z^{\prime k},\eta ^{\prime k})_{k=\overline{1,n}}$
by the rules $z^{\prime k}=z^{\prime k}(z)$ and $\eta ^{\prime k}=\frac{%
\partial z^{\prime k}}{\partial z^{l}}\eta ^{l}.$ Also, considering $T_{%
\mathbb{R}}(T^{\prime }M)$, its complexified split as $T_{\mathbb{C}%
}(T^{\prime }M)=T^{\prime }(T^{\prime }M)\oplus T^{\prime \prime }(T^{\prime
}M)$ and $V(T^{\prime }M)=\ker \pi _{\ast }\subset T^{\prime }(T^{\prime }M)$
is the vertical sub-bundle. A natural local frame for $T_{\mu }^{\prime
}(T^{\prime }M)$ is $\{\frac{\partial }{\partial z^{k}},\dot{\partial}_{k}:=%
\frac{\partial }{\partial \eta ^{k}}\}$ and the Jacobi matrix of above
transformations gives the changing rules for $\dot{\partial}_{k}$ and $\frac{%
\partial }{\partial z^{k}}.$ A complicate form of the change rule for $\frac{%
\partial }{\partial z^{k}}$ leads to the idea of complex nonlinear
connection, which refers to a complex sub-bundle $HT^{\prime }M$ in $%
T^{\prime }(T^{\prime }M)$ such that $T^{\prime }(T^{\prime }M)=HT^{\prime
}M\oplus VT^{\prime }M$, and $H_{\mu }T^{\prime }M$ is locally spanned by an
adapted frame $\{\delta _{k}:=\frac{\partial }{\partial z^{k}}-N_{k}^{j}\dot{%
\partial}_{j}\}$, (and $\{\delta _{k},\dot{\partial}_{k}\}$ on $T_{\mu
}^{\prime }(T^{\prime }M)),$ where $N_{k}^{j}(z,\eta )$ are local
coefficients of a complex nonlinear connection. By complex conjugation an
adapted frame $\{\delta _{\bar{k}},\dot{\partial}_{\bar{k}}\}$ for $T_{\mu
}^{\prime \prime }(T^{\prime }M)$ is also valid. For more details see \cite%
{Mub}.

A \textit{complex Finsler space} is a pair $(M,F)$, where $F:T^{\prime
}M\rightarrow \mathbb{R}^{+}$ is a continuous function satisfying the
following conditions:

\textit{i)} $F$ is smooth on $T^{\prime }M\backslash 0;$

\textit{ii)} $F(z,\eta )\geq 0$, the equality holds if and only if $\eta =0;$

\textit{iii)} $F(z,\lambda \eta )=|\lambda |F(z,\eta )$\ \  $\forall \ \lambda
\in \mathbb{C}$;

\textit{iv)} the Hermitian matrix $\left( g_{i\bar{j}}(z,\eta )\right) $ is
positive definite, where $g_{i\bar{j}}:=\frac{\partial ^{2}F^{2}}{\partial
\eta ^{i}\partial \bar{\eta}^{j}}$ is the fundamental metric tensor.
Equivalently, this means that the corresponding indicatrix $(F$-indricatrix)
is strongly pseudoconvex.

In complex Finsler geometry, Chern-Finsler complex nonlinear connection
(with the local coefficients $N_{j}^{i}:=g^{\overline{m}i}\frac{\partial g_{l%
\overline{m}}}{\partial z^{j}}\eta ^{l}$) is the main tool of study (cf. \cite{A-P}), From now on, by $\delta _{k}$ we denote the adapted frame
with respect to it. Also, Chern-Finsler complex nonlinear connection induces
a complex spray $S=\eta ^{k}\frac{\partial }{\partial z^{k}}-2G^{k}(z,\eta )%
\dot{\partial}_{k},$ where $2G^{k}=N_{j}^{k}\eta ^{j}.$ By a complex spray $S$
 the \textit{generalized Berwald }spaces are characterized, namely $(M,F)$ is generalized
Berwald iff $G^{i}$ are holomorphic functions with respect to $\eta $, i.e.
$\dot{\partial}_{\bar{h}}G^{i}=0$ (cf. \cite{Al-Mu3, Al-Mu4}).

In \cite{A-P}'s terminology the complex Finsler space $(M,F)$ is \textit{K%
\"{a}hler}$\;$iff $T_{jk}^{i}\eta ^{j}=0$ and \textit{weakly K\"{a}hler }iff%
\textit{\ } $g_{i\overline{l}}T_{jk}^{i}\eta ^{j}\overline{\eta }^{l}=0,$
where $T_{jk}^{i}:=L_{jk}^{i}-L_{kj}^{i}$ and $L_{jk}^{i}=\dot{\partial}%
_{j}N_{k}^{i}.$ We notice that in the particular case of the complex Finsler
metrics which come from Hermitian metrics on $M,$ so-called \textit{purely
Hermitian metrics} in \cite{Mub}, i.e. $g_{i\overline{j}}=g_{i\overline{j}%
}(z)$$,$ both kinds of K\"{a}hler are the same. From now on, a purely
Hermitian metric will be called simply Hermitian metric.

A holomorphic curvature of a complex Finsler space is the analogue of
a holomorphic sectional curvature from Hermitian geometry \cite{A-P}%
. Due to the result given in Lemma 2.1 of \cite{Al-Mu4} the holomorphic
curvature of $F$ in direction $\eta $ can be expressed in terms of spray
coefficients, that is
\begin{equation}
\mathcal{K}_{F}(z,\eta ):=-\frac{4}{F^{4}}g_{k\bar{m}}\frac{\partial G^{k}}{%
\partial \bar{z}^{h}}\bar{\eta}^{h}\bar{\eta}^{h},  \label{II.2}
\end{equation}%
and it depends both on the position $z\in M$ and the direction $\eta .$

Two complex Finsler metrics $F$ and $\tilde{F}$ on a common underlying
manifold $M$ are called\textit{\ conformal }if and only if\textit{\ }$\tilde{%
F}^{2}=\rho (z)F^{2}$ (equivalently, $\tilde{g}_{i\bar{j}}=\rho (z)g_{i\bar{%
j}})$, with a smooth, positive and real valued function $\rho (z)$, cf. \cite{A}. If $\ \rho =const.$, then $F$ and $\tilde{F}$ will be called \textit{%
homothetic}.

An important class of complex Finsler metrics is given by the \textit{%
complex Randers metrics}. Namely, considering $z\in M,$ $\eta \in
T_{z}^{\prime }M,$ $\eta =\eta ^{i}\frac{\partial }{\partial z^{i}}$, $%
a:=a_{i\bar{j}}(z)dz^{i}\otimes d\bar{z}^{j}$ a Hermitian positive metric
and $b=b_{i}(z)dz^{i}$ a differential $(1,0)$-form in \cite{A-M},
a \textit{complex Randers metric} is defined on $T^{\prime }M$ by $F(z,\eta ):=\alpha
(z,\eta )+|\beta (z,\eta )|,$ where $\alpha (z,\eta ):=\sqrt{a_{i\bar{j}%
}(z)\eta ^{i}\bar{\eta}^{j}},$ $|\beta (z,\eta )|=\sqrt{\beta (z,\eta )%
\overline{\beta (z,\eta )}}$ and$\;\beta (z,\eta ):=b_{i}(z)\eta ^{i}$. A
strongly pseudoconvexity property is determined by $||b||^{2}\in (0,1],$
where $||b||^{2}:=a^{\bar{j}i}b_{i}b_{\bar{j}}.$

\subsection{Pojectively related complex Finsler metrics}

Let $\gamma :[0,1]\rightarrow M,$ $\gamma (t)=(\gamma ^{k}(t))=(z^{k}(t)),$ %
$k=\overline{1,n},$ with  a real parameter $t$ and $\frac{d\gamma ^{k}}{dt}=\eta
^{k}(t)$, be a geodesic curve corresponding to the complex Finsler metric $%
F $ on $M.$ In local coordinates it satisfies the equations (see \cite{A-P}%
, p. 101; \cite{Mub}):%
\begin{equation}
\frac{d^{2}\gamma ^{k}}{dt^{2}}+2G^{k}\left(\gamma (t),\frac{d\gamma }{dt}%
\right)=\theta ^{\ast k}\left(\gamma (t),\frac{d\gamma }{dt}\right);\;k=\overline{1,n},
\label{XXX}
\end{equation}%
where $\theta ^{\ast k}=g^{\bar{m}k}g_{h\bar{p}}(L_{\bar{j}\bar{m}}^{\bar{p}%
}-L_{\bar{m}\bar{j}}^{\bar{p}})\eta ^{h}\bar{\eta}^{j},$ with $\theta ^{\ast
k}=0$ in the weakly K\"{a}hler case. The length of the geodesic curve $%
\gamma $ is given by $l_{F}(\gamma (t)):=\int\limits_{0}^{1}F\left(\gamma (t),\frac{%
d\gamma (t)}{dt}\right)dt.$

Two complex Finsler metrics $F$ and $\tilde{F}$ on a common underlying
manifold $M$ are called \textit{projectively related} if any complex
geodesic curve, in \cite{A-P}' s sense, of the former is also complex
geodesic curve for the latter as point sets, and vice versa. This
means that between the spray coefficients $G^{i}$ and $\tilde{G}^{i}$ there
is a so-called \textit{projective change} $\tilde{G}^{i}=G^{i}+B^{i}+P\eta
^{i},$ where $P$ is a smooth function on $T^{\prime }M$ with complex values
and $B^{i}:=\frac{1}{2}(\tilde{\theta}^{\ast i}-\theta ^{\ast i}).$ The
exploration of the projective change leads us to projective curvature
invariants: three of Douglas type and two of Weyl type. Vanishing the
projective curvature invariants of Douglas type defines the complex Douglas
spaces and a projective curvature invariant of Weyl type characterizes the
complex Berwald spaces, for more details see \cite{Al-Mu3}. Also, in \cite%
{Al-Mu4} it is established that $(M,F)$ is a complex Berwald space
if and only if it is generalized Berwald space and weakly K\"{a}hler.

\begin{corollary}
\cite{Al-Mu2} Let $F$ be a generalized Berwald metric on the manifold $M$
and $\tilde{F}$ another complex Finsler metric on $M.$ Then, $F$ and $\tilde{%
F}$ are projectively related if and only if
\begin{eqnarray}
\dot{\partial}_{\bar{r}}\left(\delta _{k}\tilde{F}\right)\eta ^{k} &=&\frac{1}{\tilde{F}%
}\left(\delta _{k}\tilde{F}\right)\eta ^{k}\left(\dot{\partial}_{\bar{r}}\tilde{F}%
\right)\;;\;B^{r}=-\frac{1}{\tilde{F}}\theta ^{\ast l}\left(\dot{\partial}_{l}\tilde{F}%
\right)\eta ^{r} \;;\;P =\frac{1}{\tilde{F}}\left[\left(\delta _{k}\tilde{F}\right)\eta ^{k}+\theta ^{\ast i}\left(%
\dot{\partial}_{i}\tilde{F}\right)\right],  \label{II'}
\end{eqnarray}
for any $r=\overline{1,n}.$ Moreover, the projective change is $\tilde{G}%
^{i}=G^{i}+\frac{1}{\tilde{F}}(\delta _{k}\tilde{F})\eta ^{k}\eta ^{i}$ and $%
\tilde{F}$ is also generalized Berwald.
\end{corollary}

\noindent Notice that any complex Berwald space is a complex Douglas space
and any Hermitian metric is generalized Berwald and complex Douglas. A
projective curvature invariant of Weyl type%
\begin{equation}
W_{j\bar{k}h}^{i}=K_{j\bar{k}h}^{i}-\frac{1}{n+1}\left(K_{\bar{k}j}\delta
_{h}^{i}+K_{\bar{k}h}\delta _{j}^{i}\right)  \label{120}
\end{equation}%
corresponding to a complex Berwald metric $F,$ where $K_{j\bar{k}%
h}^{i}=-\delta _{\bar{k}}L_{jh}^{i}$ and $K_{\bar{k}h}:=K_{i\bar{k}h}^{i}$,
is obtained in \cite{Al-Mu4}. If $n=1,$ then it is ever vanishing. For $n\geq
2 $ we have the following result

\begin{theorem}
\cite{Al-Mu4} Let $(M,F)$ be a connected complex Berwald space of complex
dimension $n\geq 2.$ Then, $W_{j\bar{k}h}^{i}=0$ if and only if $K_{\bar{m}j%
\bar{k}h}=\frac{\mathcal{K}_{F}}{4}(g_{j\bar{k}}g_{h\bar{m}}+g_{h\bar{k}}g_{j%
\bar{m}}),$ where $K_{\bar{r}j\bar{k}h}:=K_{j\bar{k}h}^{i}g_{i\bar{r}}.$ In
this case, $\mathcal{K}_{F}=c,$ where $c$ is a constant on $M$ and the space
is either Hermitian with $K_{\bar{k}j}=\frac{c(n+1)}{4}g_{j\bar{k}}$ or non
Hermitian with $c=0$ and $K_{j\bar{k}h}^{i}=0.$
\end{theorem}

Let $\tilde{F}$ be a locally Minkowski complex Finsler metric on the
underlying manifold $M,$ i.e. in any point of $M$ there exist local charts
in which the fundamental metric tensor $\tilde{g}_{i{\bar{j}}}$ depends only
on $\eta .$ Thus, the spray coefficients $\tilde{G}^{i}$ and the functions $%
\tilde{\theta}^{\ast i}$ associated to $\tilde{F}$ vanish in such local
charts, and then the geodesic curves are straight lines. In \cite{Al-Mu4} a
complex Finsler metric, which is projectively related to the locally
Minkowski metric $\tilde{F}$, is called \textit{locally projectively flat}.
The main properties of a locally projectively flat metric stated in \cite%
{Al-Mu4} are as follows

\begin{corollary}
\cite{Al-Mu4}  Let $(M,F)$ be a complex Finsler space. If $F$ is locally projectively flat,
then it is a complex Berwald metric with $W_{j\bar{k}h}^{i}=0$ and $G^{i}=%
\frac{1}{F}\frac{\partial F}{\partial z^{k}}\eta ^{k}\eta ^{i}.$ Moreover,
if $n\geq 2$, its holomorphic curvature in direction $\eta $ is constant,
ever vanish if $F$ is a non Hermitian metric.
\end{corollary}

\noindent Some refinements of the above results on locally projectively flat metrics
can still be made. Indeed, if $F$ is a complex Berwald metric on domain $D$
from $\mathbb{C}^{n}$ with $G^{i}=\frac{1}{F}\frac{\partial F}{\partial z^{k}%
}\eta ^{k}\eta ^{i}$, then according to Theorem 3.7 from \cite{Al-Mu2}, it
is projectively related with standard Euclidean metric on $D,$ and so it is
locally projectively flat. Conversely, if $F$ is locally projectively flat,
then by the last Corollary it is complex Berwald and $G^{i}=\frac{1}{F}%
\frac{\partial F}{\partial z^{k}}\eta ^{k}\eta ^{i}.$ So, we can thus
summarize

\begin{theorem}
Let $F$ be a complex Finsler metric on domain $D$ from $\mathbb{C}^{n}.$
Then, $F$ is locally projectively flat if and only if it is complex Berwald
and $G^{i}=\frac{1}{F}\frac{\partial F}{\partial z^{k}}\eta ^{k}\eta ^{i}.$
\end{theorem}

\section{Towards generalization}

\bigskip Let $M$ be an $n$-dimensional complex manifold and $z\in M$ the
base point of the tangent vectors $\eta \in T_{z}^{\prime }M,$ $\eta =\eta
^{i}\frac{\partial }{\partial z^{i}}.$ Considering the background Hermitian
metric $h:=h_{i\bar{j}}(z)dz^{i}\otimes d\bar{z}^{j}$ on $M$, the norm of $%
\eta $ is $||\eta ||_{h}:=\sqrt{h(\eta ,\bar{\eta})}=\sqrt{h_{i\bar{j}%
}(z)\eta ^{i}\bar{\eta}^{j}}$ and it means the necessary time trajectory of
a ship sailing on a sea represented by Hermitian manifold $(M,h)$, with the
presence of a wind determined by a tangent vector field $W\in T_{z}^{\prime
}M,$ $W=W^{i}\frac{\partial }{\partial z^{i}}.$ From this point of view the
pair $(h,W)$ will be called \textit{navigation data} (or \textit{Zermelo structure}) and, setting a spatial
function $||u(z)||_{h}$ on $M$, by $(h,||u||_{h},W)$ we mean \textit{%
generalized navigation data} (or \textit{generalized Zermelo structure}).


\subsection{Formula for modified Minkowski complex norm $F$}

In contrast to \cite{nicoleta} we do not start with the fact $||u||_{h}=%
\sqrt{h(u,\bar{u})}=1=const.$ So far $W$ indicated the mild breeze with $%
||W||_{h}\in \lbrack 0,1)$. Now, we introduce the new assumption $%
||u||_{h}\in (||W||_{h},1]$. 
Both wind and ship's own speed are space-dependent. In what follows we
assume that $v=u+W$ instead of $v=u-W$ which was used in \cite{nicoleta}.
Therefore, into $h(u,\bar{u})$ we substitute $u=v{-}W$. Then $\text{Re}h(v,%
\bar{W})=||v||_{h}||W||_{h}\cos \theta ,$ where \ $\theta \equiv
\measuredangle \{v,W\}$ and%
\begin{equation*}
f(z)=:||u||_{h}=\sqrt{h(u,\bar{u})}=\sqrt{h(v-W,\bar{v}-\bar{W})}%
=||v-W||_{h}=\sqrt{||v||_{h}^{2}-2\text{Re}h(v,\bar{W})+||W||_{h}^{2}},
\end{equation*}%
where $f:M\rightarrow (||W||_{h},1]$ is a smooth, positive, real valued
function which depends on $z$ and so, on $\bar{z}.$ Denoting $\tilde{%
\varepsilon}=||u||_{h}^{2}-||W||_{h}^{2}$ we thus get $%
||v||_{h}^{2}-2||v||_{h}||W||_{h}\cos \theta -\tilde{\varepsilon}=0$. Since $%
||W||_{h}<||u||_{h}$ the resultant $v$ is always positive, hence $%
||v||_{h}>0 $. Solving the resulting quadratic equation and choosing the
root that guarantees positivity yield
\begin{equation}
||v||_{h}^{2}=\text{Re}h(v,\bar{W})+\sqrt{[\text{Re}h(v,\bar{W})]^{2}+\tilde{%
\varepsilon}||v||_{h}^{2}}.  \label{3.1}
\end{equation}%
Since $F(z,v)=1$, we see that
\begin{equation}
F(z,v)=\frac{||v||_{h}^{2}}{||v||_{h}^{2}}=||v||_{h}^{2}\frac{q-p}{%
q^{2}-p^{2}}=\frac{\sqrt{[\text{Re}h(v,\bar{W})]^{2}+||v||_{h}^{2}\tilde{%
\varepsilon}}{-}\text{Re}h(v,\bar{W})}{\tilde{\varepsilon}},  \label{3.2}
\end{equation}%
where $q:=\sqrt{[\text{Re}h(v,\bar{W})]^{2}+||v||_{h}^{2}\tilde{\varepsilon}}
$ and $p:=$Re$h(v,\bar{W}).$

\begin{remark}
If $W\neq 0$, i.e. $||W||_{h}\in (0,1)$, then we can consider $\varphi =%
\text{arg }h(v,\bar{W}).$ Thus, $\text{Re}h(v,\bar{W})=|h(v,\bar{W})|\cos
\varphi $ and the above relation can be given as follows
\begin{equation}
F(z,v)=\frac{\sqrt{\left[ |h(v,\bar{W})|\cos \varphi \right]
^{2}+||v||_{h}^{2}\tilde{\varepsilon}}{-}|h(v,\bar{W})|\cos \varphi }{\tilde{%
\varepsilon}},  \label{3.3}
\end{equation}%
where $\varphi =\varphi (z)$.
\end{remark}

\noindent Next, we need to deduce $F(z,\eta )$ for an arbitrary $\eta \in T^{\prime }M$%
. Every non-zero $\eta $ is expressible as a complex multiple $\lambda \in
\mathbb{C}$ of some $v$ with $F(z,v)=1$.
We consider $ \varphi $ which is constant. This implies that $F$ is homogeneous, that is
$F(z,\eta )=F(z,\lambda v)=|\lambda |F(z,v)=|\lambda |$. Using this
homogeneity with the assumption that $\varphi $ is a constant angle and the
formula derived for $F(z,v)$ the result is
\begin{equation}
F(z,\eta )=\frac{\sqrt{[\text{Re}h(\eta ,\bar{W})]^{2}+||\eta
||_{h}^{2}(||u||_{h}^{2}-||W||_{h}^{2})}}{||u||_{h}^{2}-||W||_{h}^{2}}{-}%
\frac{[\text{Re}h(\eta ,\bar{W})]}{||u||_{h}^{2}-||W||_{h}^{2}}.  \label{3.4}
\end{equation}

\noindent  Also, in our approach the case with the absence of a wind, i.e., $W=0$ is
not neglected. Then the solution (\ref{3.4}) to the Zermelo navigation
problem becomes $F(z,\eta )=\frac{1}{||u||_{h}}||\eta ||_{h}.$ Therefore,
when a wind vanishes the resulting purely Hermitian metric is conformal to
the background metric $h$. {It is clear that for $||u||_{h}=1$ the solution
is given by $F(z,\eta )=||\eta ||_{h}=\sqrt{h_{i\bar{j}}\eta ^{i}\bar{\eta}%
^{j}}.$ }

{Now, if $W\neq 0$, i.e., $||W||_{h}\in (0,1)$ then we obtain}%
\begin{equation}
F(z,\eta )=\frac{\sqrt{|h(\eta ,\bar{W})|^{2}\cos ^{2}\varphi +||\eta
||_{h}^{2}(||u||_{h}^{2}-||W||_{h}^{2})}}{||u||_{h}^{2}-||W||_{h}^{2}}{-}%
\frac{|h(\eta ,\bar{W})|\cos \varphi }{||u||_{h}^{2}-||W||_{h}^{2}}.
\label{ran_mira}
\end{equation}

Let $||W||_{h}\in (0,1)$ and $\varphi $ be a constant angle such that $\cos
\varphi <0$. The metric \eqref{ran_mira} is of complex Randers type. By
hypothesis, $||W||_{h}<||u||_{h}$, hence $\tilde{\varepsilon}>0$. Thus, the
formula for $F(z,\eta )$ is positive whenever $\eta \neq 0$ and $F(z,0)=0$.
Also, $F$ is complex homogeneous with respect to $\eta $, $F(z,\lambda \eta
)=|\lambda |F(z,\eta )$, where $\lambda \in \mathbb{C}$ and smooth on $%
T^{\prime }M\setminus 0$. The resulting complex Randers metric is composed
of the modified new Hermitian metric and $(1,0)$-form in comparison to the
analogous terms for the case of constant unit ship's speed through the water
investigated in \cite{nicoleta}. {Likewise, if $\varphi $ is a constant
angle such that $\cos \varphi >0$ then \eqref{ran_mira} stands for complex $%
(\alpha ,\beta )$- functions, i.e. $F=\alpha -|\beta |$, however they are
not of complex Randers type }but under some additional assumptions, they can
be complex Finsler metrics.

\begin{remark}
The complex homogeneity of \eqref{ran_mira} is conditioned by the demand for
the angle $\varphi $ to be constant. Without this additional condition, that
is $\varphi $ depends on $z,$ the resulting metric \eqref{ran_mira} is only
real positive homogeneous, i.e. $F(z,c\eta )=cF(z,\eta ),$ where $c\in
\mathbb{R}_{+}$. Hence, it cannot be a complex Finsler metric. Indeed, if $%
\eta =cv$ then $\arg h(\eta ,\bar{W})=\arg h(cv,\bar{W})=\arg c$ $h(v,\bar{W}%
)=\arg h(v,\bar{W})$ for any $c\in \mathbb{R}_{+}\mathbf{.}$
\end{remark}

Further on, we pay more attention to two particular values of $\cos \varphi
, $ namely $\cos \varphi =0$ and $\cos \varphi =-1$, with $W\neq 0.$ The
condition $\cos \varphi =0$ implies that Re$h(v,\bar{W})=0$ and so, $\cos
\theta =\frac{\text{Re}h(v,\bar{W})}{||v||_{h}||W||_{h}}=0$. This means that
$v$ and $W$ are orthogonal. Conversely, if $v$ and $W$ are orthogonal then Re%
$h(v,\bar{W})=0$. We thus get Re$h(v,\bar{W})=\left\vert h(v,\bar{W}%
)\right\vert \cos \varphi =0, $what gives $\cos \varphi =0.$ So, we have
proved

\begin{corollary}
$\cos \varphi =0$ if and only if $v$ and $W$ are orthogonal.
\end{corollary}

\noindent Therefore, if $v$ and $W$ are orthogonal then the solution \eqref{ran_mira}
is a purely Hermitian metric
\begin{equation}
F(z,\eta )=\frac{1}{\sqrt{||u||_{h}^{2}-||W||_{h}^{2}}}||\eta ||_{h}=\frac{1%
}{\sqrt{||u||_{h}^{2}-||W||_{h}^{2}}}\sqrt{h_{i\bar{j}}\eta ^{i}\bar{\eta}%
^{j}},  \label{3.5}
\end{equation}%
which is conformal to the background Hermitian metric $h$. Then $\frac{1}{%
\sqrt{||u||_{h}^{2}-||W||_{h}^{2}}}$ depends on $z$ and plays the role of a
conformal factor. Thus, $g_{i\bar{j}}(z,\eta ):=a_{i\bar{j}}(z)=\frac{1}{%
\tilde{\varepsilon}}h_{i\bar{j}}(z)$. A necessary and sufficient condition
for $\cos \varphi =-1$ can be formulated as follows

\begin{corollary}
$\cos \varphi =-1$ if and only if $h(v,\bar{W})$ is negative real valued,
i.e. $h(v,\bar{W})<0$.
\end{corollary}

\begin{proof}
We have $h(v,\bar{W})=\rho ($ $\cos \varphi +i\sin \varphi ),$ where $\rho
:=\left\vert h(v,\bar{W})\right\vert >0.$ If $\cos \varphi =-1$ then $\sin
\varphi =0$ and $h(v,\bar{W})=-\rho $ which is negatively real valued.
Conversely, if $h(v,\bar{W})$ is negatively real valued then $\varphi =\arg
h(\eta ,\bar{W})=\pi$ and so, $\cos \varphi =-1.$
\end{proof}

\begin{remark}
Collinearity of the tangent vectors $v$ and $W$ implies $\cos \varphi =-1$
but the converse is not true (see example in \cite{nicoleta} with the
setting $v=u+W$).
\end{remark}

\subsection{Modified Hermitian metric and $(1,0)$-form}

The resulting complex Randers metric $F$ can also be presented in the form $%
F=\alpha +|\beta |$ as the sum of two components. Explicitly,

\begin{itemize}
\item the first term is the norm of $\eta $ with respect to a new Hermitian
metric $a$
\begin{equation}
\alpha (z,\eta )=\sqrt{a_{i\bar{j}}(z)\eta ^{i}\bar{\eta}^{j}},\quad \text{
where }\quad a_{i\bar{j}}=\frac{h_{i\bar{j}}}{\tilde{\varepsilon}}+\frac{%
W_{i}}{\tilde{\varepsilon}}\frac{\bar{W}_{j}}{\tilde{\varepsilon}}\cos
^{2}\varphi ,  \label{alpha_complex}
\end{equation}

\item the second term is the value on $\eta $ of a differential $(1,0)$-form
$b$
\begin{equation}
{|\beta (z,\eta )|=\sqrt{\beta (z,\eta ){\overline{\beta (z,\eta )}}},}\quad
\beta (z,\eta )=b_{i}(z)\eta ^{i},\quad \text{where}\quad b_{i}={-}\frac{%
W_{i}}{\tilde{\varepsilon}}\cos \varphi ,  \label{beta_complex}
\end{equation}%
where $W_{i}=h_{i\bar{j}}\bar{W}^{j}$ and $\tilde{\varepsilon}%
=||u(z)||_{h}^{2}-W^{i}W_{i}=||u(z)||_{h}^{2}-h(W,\bar{W}%
)=||u||_{h}^{2}-||W||_{h}^{2}.$
\end{itemize}

\noindent The inverse of $a_{i\bar{j}}(z)$ from \eqref{alpha_complex} is $a^{\bar{j}i}=%
\tilde{\varepsilon}\left( h^{\bar{j}i}-\frac{\cos ^{2}\varphi }{%
||u||_{h}^{2}-||W||_{h}^{2}\sin ^{2}\varphi }W^{i}\bar{W}^{j}\right) $ and
so,

$b^{i}:=a^{\bar{j}i}b_{\bar{j}}=\tilde{\varepsilon}\left(h^{\bar{j}i}-\frac{\cos
^{2}\varphi }{||u||_{h}^{2}-||W||_{h}^{2}\sin ^{2}\varphi }W^{i}\bar{W}%
^{j}\right)\left(-\frac{W_{\bar{j}}\cos \varphi }{\tilde{\varepsilon}}\right)=\frac{-\tilde{%
\varepsilon}\cos \varphi }{||u||_{h}^{2}-||W||_{h}^{2}\sin ^{2}\varphi }%
W^{i},$

$||b||^{2}:=b_{i}b^{i}=\frac{||W||_{h}^{2}\cos ^{2}\varphi }{%
||u||_{h}^{2}-||W||_{h}^{2}\sin ^{2}\varphi }<1,$ because $%
0<||W||_{h}<||u||_{h}\leq 1$.  This guarantees the strongly pseudoconvexity
of the function $F(z,\eta )$ from \eqref{ran_mira} on $F$-indicatrix. It
means that the fundamental tensor $g_{i\bar{j}}=\dot{\partial}_{i}\dot{%
\partial}_{\bar{j}}F^{2}$ is positive definite.
Taking into account the above notations, we can summarize the obtained
solutions of the generalized Zermelo navigation problem (briefly,
generalized ZNP) in the following way

\begin{proposition}
Let $(M,h)$ be a Hermitian manifold. Generalized navigation data\textit{\ }$%
(h,||u||_{h},W)${$,$} with $||W||_{h}\in \lbrack 0,1),$ $||u||_{h}\in
(||W||_{h},1]$ induce on the holomorphic tangent bundle $T^{\prime }M$
the following \

i) If $||W||_{h}\in (0,1)$ and $\varphi $ is a constant angle such that $%
\cos \varphi <0,$ then the solution of generalized ZNP is the complex
Randers metric $F(z,\eta )=\alpha +|\beta |$.

ii) If $||W||_{h}\in (0,1)$ and $\cos \varphi =0,$ then the solution of
generalized ZNP is the Hermitian metric (\ref{3.5}) which is conformal to $h$
and $a_{i\bar{j}}=\frac{1}{\tilde{\varepsilon}}h_{i\bar{j}}$. \

iii) If $W=0,$ then the solution of generalized ZNP is the Hermitian metric $%
F(z,\eta )=\frac{1}{||u(z)||_{h}}\sqrt{h_{i\bar{j}}\eta ^{i}\bar{\eta}^{j}},$
which is conformal to $h$ and $a_{i\bar{j}}=\frac{1}{||u||_{h}^{2}}h_{i\bar{j%
}}$.
\end{proposition}

\noindent For $||u||_{h}=1$ the formulae \eqref{alpha_complex} and %
\eqref{beta_complex} lead to the Hermitian metric $a$ and $(1,0)$-form $b,$
with $\varepsilon =1-||W||_{h}^{2}\geq \tilde{\varepsilon}$ according to
\cite{nicoleta}, respectively, in the case of the original Zermelo
navigation problem on Hermitian manifolds $(M,h)$, under a mild perturbation
with $||W||_{h}<1$. In this case, if $W$ vanishes then, due to iii), $%
F(z,\eta )=\sqrt{h_{i\bar{j}}\eta ^{i}\bar{\eta}^{j}}$ and so the
time-minimal solutions are represented by the geodesics of $h$, ($h$%
-geodesics) which coincide with $F$-geodesics as point sets and lengths ($h$%
-length is the same as $F$-length, formally we can write $F=h$).

Also, Proposition 3.6 gives some extensions and refinements of the previous
results in the standard setting (cf. \cite{nicoleta}). In other words we can
say that in the new modified scenario with the presence of $W$ and with $%
||u(z)||_{h}\in (||W||_{h},1)$ the time-efficient paths are the geodesics
of the solution $F$ of ZNP obtained in i) or ii). {{As we shall see in the
next section, with some additional assumptions }$h$-geodesics and $F$%
-geodesics can be the same as point sets, i.e. projectively related, however
their corresponding $h$-length and $F$-length are different.} In interpreting we can say that a ship traces the same route in the presence and absence of a wind, but her resulting speed differs. Without a perturbation ($W=0$) the geodesics of the background
Hermitian metric $h$ are not necessarily the solutions to the problem as it
is in the standard case with $||u||_{h}=1$. {The difference is made by the
influence of the new factor $||u(z)||_{h}$ which now becomes a spatial
function of $z$ and the resulting Hermitian metric is conformal to the
background metric $h$, but $F\neq h$ for }$||u(z)||_{h}${$\neq 1$. Since $%
g_{i\bar{j}}=\frac{1}{||u(z)||_{h}^{2}}h_{i\bar{j}}$, $\frac{1}{%
||u(z)||_{h}^{2}}$ plays the role of a conformal factor. }

{Next, we focus on the implications ii) and iii) from the above Proposition.
By generalized Zermelo navigation both of them produce Hermitian metrics
which are conformal to the background metric $h,$ however with different
conformal factors. So, it is natural to ask whether we can obtain any
Hermitian metric $F$ $\ $which is conformal to $h$ as a solution to
generalized ZNP. The answer comes below }

\begin{corollary}
Let $F$ be Hermitian metric which is conformal to background Hermitian
metric $h,$ i.e. $F(z,\eta )=\rho (z)||\eta ||_{h},$ with a smooth, positive
and real valued function $\rho (z)$. Then $F$ is a solution of generalized
ZNP if and only if either $\rho (z)=\frac{1}{\sqrt{\tilde{\varepsilon}}}$
and $\cos \varphi =0$ or $\rho (z)=\frac{1}{||u||_{h}}$ and $W=0$.
\end{corollary}

\begin{proof} If $F(z,\eta )=\rho (z)||\eta ||_{h}$ stands for a solution of
generalized ZNP, then according to formula (\ref{3.4}) we have
\begin{equation*}
\frac{\sqrt{\lbrack \text{Re}h(\eta ,\bar{W})]^{2}+||\eta
||_{h}^{2}(||u||_{h}^{2}-||W||_{h}^{2})}}{||u||_{h}^{2}-||W||_{h}^{2}}{-}%
\rho (z)||\eta ||_{h}-\frac{[\text{Re}h(\eta ,\bar{W})]}{%
||u||_{h}^{2}-||W||_{h}^{2}}=0
\end{equation*}%
which contains an irrational part and a rational one. We can deduce that Re$%
h(\eta ,\bar{W})=0$ and

\noindent$\left[\frac{1}{\sqrt{||u||_{h}^{2}-||W||_{h}^{2}}}{-}%
\rho (z)\right]||\eta ||_{h}=0.$ The first relation implies that either $\cos
\varphi =0$ with $W\neq 0$ or $W=0.$ These together with the second relation
lead to either $\rho (z)=\frac{1}{\sqrt{\tilde{\varepsilon}}}$ or $\rho (z)=%
\frac{1}{||u||_{h}}.$ The converse results come by Proposition 3.6.
\end{proof}

The perturbation of a Hermitian metric $h$ by a vector field $W$ with
the introduction of a spatial function $||u(z)||_{h}$, where $%
0<||W||_{h}<||u||_{h}\leq 1,$ i.e. generalized navigation data\textit{\ }$%
(h,||u||_{h},W),$ and the additional assumption $\cos \varphi =-1$
generate strongly pseudoconvex Randers metrics. An inverse problem asks if
every complex Randers metric $F(z,\eta )=\alpha +\left\vert \beta
\right\vert $ can be realized through the perturbation of some Hermitian
metric $h$ by some vector field $W$ satisfying $0<h(W,\bar{W})<h(u,\bar{u})$
and taking into account $f(z):=||u(z)||_{h}$. It can be checked by
constructing $h$ and rescaled wind $\tilde{W}$ that perturbing the above $h$
by the stipulated $\tilde{W}$ gives back the complex Randers metric we
started with. For clarity we give the proof, although it is likewise the
standard formulation of the complex ZNP, however we include some
refinements. Considering the complex metric $F(z,\eta )=\alpha +\left\vert
\beta \right\vert ,$ with $\alpha =\sqrt{a_{i\bar{j}}(z)\eta ^{i}\bar{\eta}%
^{j}},$ $\beta =b_{i}(z)\eta ^{i},$ $b^{i}:=a^{\bar{j}i}b_{\bar{j}}$ , $%
||b||^{2}:=b^{i}b_{i}<1$ and $\tilde{\omega}:=f^{2}(z)(1-||b||^{2}),$ we
construct $h$ and $W$ in the following manner%
\begin{equation}
h_{i\bar{j}}(z)=\tilde{\omega}(a_{i\bar{j}}-b_{i}b_{\bar{j}});\;{%
||u(z)||=f(z)};\;W^{i}(z)=\frac{f^{2}(z)b^{i}}{\tilde{\omega}}{.}  \label{X}
\end{equation}

\noindent By straightforward computations we obtain $||W||_{h}^{2}:=h_{i\bar{%
j}}W^{i}\bar{W}^{j}=\tilde{\omega}(a_{i\bar{j}}-b_{i}b_{\bar{j}})\frac{%
f^{2}(z)b^{i}}{\tilde{\omega}}\frac{f^{2}(z)\bar{b}^{j}}{\tilde{\omega}}%
=f^{2}(z)||b||^{2}.$ Therefore, from $||W||_{h}=f(z)||b||$ with $||b||<1$ we obtain $||W||_h<f(z)$, since $||b||=\frac{||W||_h}{f(z)}<1$. Hence, $\tilde{\varepsilon}%
=||u||_{h}^{2}-||W||_{h}^{2}=f^{2}(z)-f^{2}(z)||b||^{2}=\tilde{\omega}$ and $%
W_{i}:=h_{i\bar{j}}\bar{W}^{j}=\tilde{\omega}(a_{i\bar{j}}-b_{i}b_{\bar{j}})%
\frac{f^{2}(z)\bar{b}^{j}}{\tilde{\omega}}=f^{2}(z)b_{i}(1-||b||^{2})=\tilde{%
\omega}b_{i}.$ Thus, $h(\eta ,\bar{W})=h_{i\bar{j}}\eta ^{i}\bar{W}%
^{j}=W_{i}\eta ^{i}=\tilde{\omega}\beta .$ Due to {\eqref{ran_mira}}, the
generalized navigation data\textit{\ }$(h,||u||_{h},W)$ from (\ref{X}) lead
to the function
\begin{equation*}
\tilde{F}=\tilde{\alpha}+|\tilde{\beta}|=\sqrt{\alpha ^{2}-\left\vert \beta
\right\vert ^{2}\sin ^{2}\varphi }\mathbf{-}\left\vert \beta \right\vert
\cos \varphi
\end{equation*}%
with $\tilde{\alpha}(z,\eta )=\sqrt{\tilde{a}_{i\bar{j}}(z)\eta ^{i}\bar{\eta%
}^{j}},$ $\;\tilde{\beta}(z,\eta )=\tilde{b}_{i}(z)\eta ^{i},$ where $\tilde{%
b}_{i}=b_{i}\cos \varphi $ and $\tilde{a}_{i\bar{j}}(z)=a_{i\bar{j}}-(1-\cos
^{2}\varphi )b_{i}b_{\bar{j}}.$ The obtained function $\tilde{F}$ is complex
homogeneous with respect to $\eta $ if and only if $\varphi $ is a constant
angle. Thus, we consider the following cases:

\noindent\textbf{Case I.} $\tilde{F}=F$ if and only if $\cos \varphi =-1,$
i.e. by (\ref{X}) we obtain the same complex Randers metric.

\noindent \textbf{Case} \textbf{II.} $\tilde{F}=\frac{1}{\sqrt{\tilde{%
\varepsilon}}}h,$ where $h:=\sqrt{h_{{i}\bar{j}}\eta ^{{i}}\bar{\eta}^{j}},$ if
and only if $\cos \varphi =0.$ So, by (\ref{X}) we obtain the Hermitian
metric $\tilde{a}_{i\bar{j}}(z)=a_{i\bar{j}}-b_{i}b_{\bar{j}}=\frac{1}{%
\tilde{\varepsilon}}h_{{i}\bar{j}}$ if and only if $\cos \varphi =0.$

\noindent Notice that if $\varphi $ is a constant angle such that {$\cos \varphi \neq
0,-1,$ then }$\tilde{F}<F$ and, {under additional assumption the function }$%
\tilde{F}$ can be $(\alpha ,\beta )$-complex Finsler function. Summarizing,
we obtain\

\begin{theorem}
i) A complex Finsler metric $F$ is of complex Randers type, i.e., it has the
form $F(z,\eta )=\alpha +|\beta |$, if and only if it solves the generalized
Zermelo navigation problem on some Hermitian manifold $(M,h)$, with a
variable in space ship's speed $||u||_{h}$, under the influence of a wind $W$
which satisfy $0<||W||_{h}<||u||_{h}\leq 1$ and $\cos \varphi =-1$.

ii) A complex Finsler metric $F$ is Hermitian conformal to $h$ if and only
if it solves the generalized Zermelo navigation problem on some Hermitian
manifold $(M,h)$, with a variable in space ship's speed $||u||_{h}$, under
the influence of a wind $W$ which satisfy either $0<||W||_{h}<||u||_{h}\leq
1 $ and $\cos \varphi =0$ or $W=0$. Also, $F$ is Hermitian homothetic to $h$
if and only if either $\cos \varphi =0$ and $\tilde{\varepsilon}%
=||u||_{h}^{2}-||W||_{h}^{2}=const$. or $W=0$ and $||u||_{h}$ $=const.$
\end{theorem}

\noindent We can ask whether decreasing a ship's own speed $||u||_{h}$ under fixed
wind field $W$ will cause the same effect on the time-minimal path as
increasing the wind "force" with $||u||_{h}=1$ and maintaining the same
relation $\frac{||W||_{h}}{||u||_{h}}$. Since $0<||u||_{h}<1$ the decrease
of a ship's velocity introduces a larger effective wind. We have a
straightforward correspondence with the standard setting, (i.e. with $%
||u||_{h}=1).$ Indeed, {having already solved generalized ZNP and }replacing
$W^{i}$ with the normalized wind $\tilde{W}^{i}=\frac{1}{||u(z)||_{h}}W^{i}$%
, where $W\neq 0$ and $u^{i}$ with $\tilde{u}^{i}=\frac{1}{||u(z)||_{h}}%
u^{i} $ in the generalized approach, we obtain $\tilde{\alpha}%
=||u(z)||_{h}\alpha $ and $|\tilde{\beta}|=||u(z)||_{h}|\beta |$, where the
formulae for $\tilde{\alpha}$ and $|\tilde{\beta}|$ are then the same as given in
\cite{nicoleta}.

\section{Geodesics corresponding to solutions of Zermelo's problem}

We aim to find some connections among geodesic curves corresponding to the metrics $h,$ $%
a$ and $F$ when a solution to ZNP is a complex Randers metric $F$, i.e. $%
\cos \varphi =-1,$ and then when a solution $F$ is conformal to $h,$
i.e. $W=0$ or $W\neq 0$ and $\cos \varphi =0$.

\subsection{Complex Randers metrics as the solutions to
Zermelo's problem}

The study of complex Randers metrics was developed in several papers (cf. \cite%
{Al-Mu3, Al-Mu2, Al-Mu3, Al-Mu4}), giving some applications of the general
theories regarding to generalized Berwald spaces (cf. \cite{Al-Mu3, Al-Mu4}),
complex Douglas spaces and projectively related complex Finsler metrics (cf.
\cite{Al-Mu2, Al-Mu3}). With these in mind, further on we come with some new
results from point of view of generalized ZNP.

\subsubsection{Connection between $a$ and $F$}

\bigskip Let $F=\alpha +|\beta |$ be a complex Randers metric, with $\alpha
(z,\eta )=\sqrt{a_{i\bar{j}}(z)\eta ^{i}\bar{\eta}^{j}}$ and $\beta (z,\eta
)=b_{i}(z)\eta ^{i}.$ First, we show some results which show connections
between the metrics $a$ and $F.$ For the metric $F$ the spray coefficients are
expressed as
\begin{equation}
G^{i}=\overset{a}{G^{i}}+\frac{1}{2\gamma }\left(l_{\bar{r}}\frac{\partial b^{%
\bar{r}}}{\partial z^{j}}-\frac{\beta ^{2}}{|\beta |^{2}}\frac{\partial b_{%
\bar{r}}}{\partial z^{j}}\bar{\eta}^{r}\right)\xi ^{i}\eta ^{j}+\frac{\beta }{%
4|\beta |}k^{\overline{r}i}\frac{\partial b_{\bar{r}}}{\partial z^{j}}\eta
^{j},  \label{3.0}
\end{equation}%
where $\overset{a}{G^{i}}=\frac{1}{2}\overset{a}{N_{j}^{i}}\eta ^{j}$ are
the spray coefficients of the Hermitian metric $a$ and $\gamma :=L+\alpha
^{2}(||b||^{2}-1)$, $\xi ^{i}:=\overline{\beta }\eta ^{i}+\alpha ^{2}b^{i}$,
$k^{\bar{r}i}:=2\alpha a^{\bar{j}i}+\frac{2(\alpha ||b||^{2}+2|\beta |)}{%
\gamma }\eta ^{i}\bar{\eta}^{r}-\frac{2\alpha ^{3}}{\gamma }b^{i}\bar{b}^{r}-%
\frac{2\alpha }{\gamma }(\beta \eta ^{i}b^{\bar{r}}+\beta b^{i}\bar{\eta}%
^{r})$, $b^{i}:=a^{\bar{j}i}b_{\bar{j}},$ $||b||^{2}:=a^{\bar{j}i}b_{i}b_{%
\bar{j}}.$

A necessary and sufficient condition that a connected complex Randers space
to be generalized Berwald is $A=0,$ where $A:=(\overline{\beta }l_{\bar{r}}%
\frac{\partial b^{\bar{r}}}{\partial z^{j}}+\beta \frac{\partial b_{\bar{r}}%
}{\partial z^{j}}\bar{\eta}^{r})\eta ^{j}=(\overset{a}{\delta _{k}}|\beta
|^{2})\eta ^{k},$ $\overset{a}{\delta _{k}}=\partial _{k}-\overset{a}{%
N_{k}^{i}}\dot{\partial}_{i}$ and $l_{\bar{r}}:=a_{l\bar{r}}\eta ^{l},$ (%
\cite{Al-Mu3}). Moreover, for such spaces the formula (\ref{3.0}) is
reduced to $G^{i}=\overset{a}{G^{i}}$. Based on this, by a straightforward
computation the holomorphic curvature in direction $\eta$ corresponding to
a generalized Berwald complex Randers metric $F$ can be written as%
\begin{equation}
\mathcal{K}_{F}(z,\eta )=\frac{\alpha ^{3}}{F^{3}}\mathcal{K}_{a}(z,\eta )-%
\frac{4\overline{\beta }}{F^{3}|\beta |}\frac{\partial \overset{a}{G^{l}}}{%
\partial \bar{z}^{m}}b_{l}\bar{\eta}^{m},  \label{3.0'}
\end{equation}%
where $\mathcal{K}_{a}(z,\eta )$ is the holomorphic curvature in direction $%
\eta $ which corresponds to $a.$ Also, under assumption of generalized Berwald,
another link between geometric objects $\overset{a}{\theta ^{\ast i}}$ and $%
\theta ^{\ast i},$ corresponding to $a$ and $F,$ respectively is stated by%
\begin{equation}
\theta ^{\ast i}=\overset{a}{\theta ^{\ast i}}+\frac{1}{\gamma }\Gamma _{l%
\bar{r}\bar{m}}\eta ^{l}\bar{\eta}^{r}b^{\bar{m}}\xi ^{i}-\frac{\alpha \beta
}{|\beta |}\left(\Gamma _{l\bar{r}\bar{m}}b^{l}\bar{\eta}^{r}+2\Omega _{\bar{m}}\right)\left(%
\tilde{h}^{\bar{m}i}-\frac{\overline{\beta }}{\gamma }b^{\bar{m}}\eta ^{i}\right),
\label{4.4}
\end{equation}%
where $\overset{a}{\theta ^{\ast i}}=-\Gamma _{l\bar{r}\bar{m}}a^{\bar{m}%
i}\eta ^{l}\bar{\eta}^{r},$ $\Gamma _{l\bar{r}\bar{m}}:=\frac{\partial a_{l%
\bar{m}}}{\partial \bar{z}^{r}}$ $-\frac{\partial a_{l\bar{r}}}{\partial
\bar{z}^{m}}$, $\Omega _{\bar{m}}:=\overset{a}{N_{\bar{m}}^{\bar{s}}}b_{\bar{%
s}}-\frac{\partial b_{\bar{r}}}{\partial \bar{z}^{m}}\bar{\eta}^{r}-\frac{%
\overline{\beta }^{2}}{|\beta |^{2}}\frac{\partial b_{l}}{\partial \bar{z}%
^{m}}\eta ^{l}$ and $\tilde{h}^{\bar{m}i}:=a^{\bar{m}i}-\frac{\alpha ^{2}}{%
\gamma }b^{\bar{m}}b^{i}.$

\begin{theorem}
\cite{Al-Mu3} Let $(M,F)$ be a connected complex Randers space. Then, the
following assertions are equivalent: i) $(M,F)$ is a complex Douglas space;
ii) $A=0$ and $\theta ^{\ast i}=\overset{a}{\theta ^{\ast i}}$; iii) $a$ and
$F$ are projectively related.
\end{theorem}

\noindent According to \cite{Al-Mu3}, if the complex Randers space $(M,F)$ is a
complex Douglas, then $\Omega _{\bar{m}}=-\frac{1}{2}\Gamma _{l\bar{r}\bar{m}%
}b^{l}\bar{\eta}^{r},$ $\Omega _{\bar{m}}b^{\bar{m}}=0$ and $\Gamma _{l\bar{r%
}\bar{m}}b^{\bar{m}}=0.$ Also, in addition if $a$ is K\"{a}hler, then $(M,F)$
is a complex Berwald space and $N_{j}^{i}=\overset{a}{N_{j}^{i}}$. Conversely, if it is a complex Berwald space, then by (\ref{4.4}), $a$ is K%
\"{a}hler. Moreover, if the complex Randers metric is complex Berwald, one sees immediately that $K_{j\bar{k}h}^{i}=\overset{a}{K_{j%
\bar{k}h}^{i}}$. Then by (\ref{120}), $W_{j\bar{k}h}^{i}=\overset{a}{W_{j%
\bar{k}h}^{i}},$ where $\overset{a}{W_{j\bar{k}h}^{i}}$ is projective
curvature invariant of Weyl type corresponding to $a.$

Now, taking into account Theorem 2.2, another necessary and sufficient
condition for local projective flatness can be provided.
\begin{proposition}
Let $F$ be a complex Randers metric on domain $D$ from $\mathbb{C}^{n}.$ $F$
is locally projectively flat if and only if $\ a$ and $F$ are projectively
related and $a$ is locally projectively flat. Moreover, any of these
assertions implies $\mathcal{K}_{F}=\mathcal{K}_{a}=0.$
\end{proposition}

\begin{proof} If $\ F$ is locally projectively flat, then it is
projectively related with standard Euclidean metric on $D,$ and according to
Theorem 4.4 from \cite{Al-Mu2} it results the local projective flatness
for $a$ and then the projectively related property for $a$ and $F.$
Conversely, if $a$ is locally projectively flat, then it is K\"{a}hler. This
together with the projectively related property for $a$ and $F$ give that $F$ is
complex Berwald. Now, apllying again Theorem 4.4 from \cite{Al-Mu2} gives
the local projective flatness for $F$. Moreover, since $W_{j\bar{k}%
h}^{i}=\overset{a}{W_{j\bar{k}h}^{i}}=0$ and $K_{j\bar{k}h}^{i}=\overset{a}{%
K_{j\bar{k}h}^{i}}=0,$ we also obtain $\mathcal{K}_{F}=\mathcal{K}_{a}=0.$%
\end{proof}

\subsubsection{Connection between $a$ and $h$}

In order to establish when $h$ and $F$ are projectively related, that is they
have the same geodesics as point sets, it is necessary to find some links
between $a$ and $h,$ in terms of generalized navigation data $%
(h,||u(z)||_{h},W),$ the Hermitian metric $a$ being only an intermediary
on the way to the complex Randers solution $F.$ Thus, starting with the relations
between the Hermitian metrics $a$ and $h$%
\begin{eqnarray}
a_{i\bar{j}} &=&\frac{h_{i\bar{j}}}{\tilde{\varepsilon}}+\frac{W_{i}}{\tilde{%
\varepsilon}}\frac{\bar{W}_{j}}{\tilde{\varepsilon}},\quad \text{ }b_{i}=%
\frac{W_{i}}{\tilde{\varepsilon}},\quad W_{i}=h_{i\bar{j}}\bar{W}^{j}\text{ }%
;\quad \text{ }  \label{III'} \\
a^{\bar{j}i} &=&\tilde{\varepsilon}h^{\bar{j}i}-\frac{\tilde{\varepsilon}}{%
||u||_{h}^{2}}W^{i}\bar{W}^{j},\quad \text{ }\tilde{\varepsilon}%
=||u||_{h}^{2}-||W||_{h}^{2}=f^{2}(1-||b||^{2});  \notag \\
b^{i} &:&=a^{\bar{j}i}b_{\bar{j}}=\frac{\tilde{\varepsilon}}{||u||_{h}^{2}}%
W^{i}=(1-||b||^{2})W^{i},\quad ||b||^{2}:=b_{i}b^{i}=\frac{||W||_{h}^{2}}{%
||u||_{h}^{2}};  \notag \\
h &:&=||\eta ||_{h}=\sqrt{h_{i\bar{j}}\eta ^{i}\bar{\eta}^{j}}=\tilde{%
\varepsilon}(\alpha ^{2}-|\beta |^{2}),\quad \text{ }a:=||\eta ||_{a}=\sqrt{%
a_{i\bar{j}}\eta ^{i}\bar{\eta}^{j}},  \notag
\end{eqnarray}%
after a straightforward computation, we are lead to
\begin{equation}
\overset{a}{G^{i}}=\overset{h}{G^{i}}-\frac{1}{2\tilde{\varepsilon}}\frac{%
\partial \tilde{\varepsilon}}{\partial z^{j}}\eta ^{j}\eta ^{i}+\frac{1}{2%
\tilde{\varepsilon}}W_{0}h^{\bar{m}i}\frac{\partial W_{\bar{m}}}{\partial
z^{j}}\eta ^{j}+\frac{1}{2\tilde{\varepsilon}f^{2}}\frac{\partial \bar{W}^{m}%
}{\partial z^{j}}\left(\tilde{\varepsilon}h_{0\bar{m}}+W_{0}W_{\bar{m}}-W_{0}%
\frac{\partial f^{2}}{\partial z^{j}}\right)\eta ^{j}W^{i},  \label{IV}
\end{equation}%
where the index $0$ means the contraction by $\eta $. In terms of
generalized navigation data $(h,||u||_{h},W)$, $F$ is generalized Berwald,
(i.e. $(\overset{a}{\delta _{k}}|\beta |^{2})\eta ^{k}=0),$ if and only if%
\begin{equation}
\left(W_{0}\frac{\partial W_{\bar{0}}}{\partial z^{j}}+W_{\bar{0}}\frac{\partial
W_{0}}{\partial z^{j}}\right)\eta ^{j}=2\left(\overset{a}{G^{l}}W_{l}W_{\bar{0}}+\frac{1%
}{\tilde{\varepsilon}}\frac{\partial \tilde{\varepsilon}}{\partial z^{j}}%
\eta ^{j}\right),  \label{IV''}
\end{equation}%
rewritten in an equivalent form as%
\begin{equation}
\tilde{\varepsilon}h_{0\bar{m}}W_{\bar{0}}\frac{\partial \bar{W}^{m}}{%
\partial z^{j}}\eta ^{j}=W_{0}\left( W_{\bar{0}}\frac{\partial f^{2}}{%
\partial z^{j}}-W_{\bar{0}}W_{\bar{m}}\frac{\partial \bar{W}^{m}}{\partial
z^{j}}-f^{2}\frac{\partial W_{\bar{0}}}{\partial z^{j}}\right) \eta ^{j}.
\label{IV'}
\end{equation}

\begin{corollary}
The complex Randers solution $F$ of generalized ZNP is generalized Berwald if and only
if
\begin{equation}
\overset{a}{G^{i}}=\overset{h}{G^{i}}-\frac{1}{2\tilde{\varepsilon}}\frac{%
\partial \tilde{\varepsilon}}{\partial z^{j}}\eta ^{j}\eta ^{i}.  \label{V}
\end{equation}
\end{corollary}

\begin{proof} For the direct implication we start with formula (\ref{IV'}%
), which differentiated with respect to $\bar{\eta}^{s}$ and next
contracted by $\bar{W}^{s}$ gives%
\begin{equation}
\tilde{\varepsilon}h_{0\bar{m}}\frac{\partial \bar{W}^{m}}{\partial z^{j}}%
\eta ^{j}=W_{0}\left( \frac{\partial f^{2}}{\partial z^{j}}-W_{\bar{m}}\frac{%
\partial \bar{W}^{m}}{\partial z^{j}}-\frac{f^{2}}{||W||_{h}^{2}}\frac{%
\partial W_{\bar{m}}}{\partial z^{j}}\bar{W}^{m}\right) \eta ^{j}.
\label{VI}
\end{equation}%
Now, the last relation substituted into formula (\ref{IV'}) implies $\frac{%
\partial W_{\bar{m}}}{\partial z^{j}}(\bar{\eta}^{m}-\frac{1}{||W||_{h}^{2}}%
W_{\bar{0}}\bar{W}^{m})\eta ^{j}=0.$ By differentiation with respect to $%
\bar{\eta}^{s}$ and then with $\eta ^{k},$ we deduce that $\frac{\partial
W_{\bar{m}}}{\partial z^{j}}=\frac{1}{||W||_{h}^{2}}\frac{\partial W_{\bar{s}%
}}{\partial z^{j}}\bar{W}^{s}W_{\bar{m}}.$ Substitution of the last relation
and (\ref{VI}) in (\ref{IV}) yields (\ref{V}).
Conversely, if $\overset{a}{G^{i}}=\overset{h}{G^{i}}-\frac{1}{2\tilde{%
\varepsilon}}\frac{\partial \tilde{\varepsilon}}{\partial z^{j}}\eta
^{j}\eta ^{i}$ then the formula (\ref{IV}) is%
\begin{equation}
W_{0}h^{\bar{m}i}\frac{\partial W_{\bar{m}}}{\partial z^{j}}\eta ^{j}+\frac{1%
}{f^{2}}\frac{\partial \bar{W}^{m}}{\partial z^{j}}\left(\tilde{\varepsilon}h_{0%
\bar{m}}+W_{0}W_{\bar{m}}-W_{0}\frac{\partial f^{2}}{\partial z^{j}}\right)\eta
^{j}W^{i}=0,  \notag
\end{equation}%
which contracted with $h_{i\bar{0}}$ gives (\ref{IV'}). Hence, $F$ is
generalized Berwald.
\end{proof}

\noindent Also, under generalized Berwald assumption for the complex Randers
solution $F$ of ZNP, some computation give the connection between
holomorphic curvature in direction $\eta ,$ corresponding to $h$ and $a$:%
\begin{equation}
\mathcal{K}_{h}(z,\eta )=\frac{\tilde{\varepsilon}\alpha ^{4}}{h^{4}}%
\mathcal{K}_{a}(z,\eta )+\frac{4}{\tilde{\varepsilon}h^{4}}W_{\bar{0}}W_{l}%
\frac{\partial \overset{a}{G^{l}}}{\partial \bar{z}^{m}}\bar{\eta}^{m}-\frac{%
2}{\tilde{\varepsilon}h^{2}}\left(\frac{\partial \tilde{\varepsilon}}{\partial
\bar{z}^{m}}\frac{\partial \tilde{\varepsilon}}{\partial z^{j}}-\frac{%
\partial ^{2}\tilde{\varepsilon}}{\partial \bar{z}^{m}\partial z^{j}}\right)\eta
^{j}\bar{\eta}^{m},  \label{III.0}
\end{equation}%
and taking into account (\ref{3.0'}), we have the connection between
holomorphic curvature in direction $\eta ,$ corresponding to $h,$ $a$ and $F$%
:%
\begin{equation}
\mathcal{K}_{F}(z,\eta )=\frac{1}{F^{2}\sqrt{W_{0}W_{\bar{0}}}}\left[\tilde{%
\varepsilon}\alpha ^{3}\mathcal{K}_{a}(z,\eta )-\frac{h^{4}}{F}\mathcal{K}%
_{h}(z,\eta )+\frac{2h^{2}}{\tilde{\varepsilon}F}\left(\frac{1}{\tilde{\varepsilon%
}}\frac{\partial \tilde{\varepsilon}}{\partial \bar{z}^{m}}\frac{\partial
\tilde{\varepsilon}}{\partial z^{j}}-\frac{\partial ^{2}\tilde{\varepsilon}}{%
\partial \bar{z}^{m}\partial z^{j}}\right)\eta ^{j}\bar{\eta}^{m}\right].  \label{III.0'}
\end{equation}

\noindent Under assumption (\ref{V}), by computation we obtain%
\begin{equation}
\overset{h}{\theta ^{\ast i}}=\overset{a}{\theta ^{\ast i}}-\frac{1}{\tilde{%
\varepsilon}}\frac{\partial \tilde{\varepsilon}}{\partial \bar{z}^{m}}\bar{%
\eta}^{m}\eta ^{i}-\frac{\tilde{\varepsilon}}{f^{2}}\Gamma _{0\bar{0}\bar{m}}%
\bar{W}^{m}W^{i}+\left[\frac{1}{\tilde{\varepsilon}}\frac{\partial \tilde{%
\varepsilon}}{\partial \bar{z}^{m}}h_{0\bar{0}}+W_{0}\left(\frac{\tilde{%
\varepsilon}}{f^{2}}\Gamma _{l\bar{0}\bar{m}}W^{l}+2\Omega _{\bar{m}}\right)\right]h^{%
\bar{m}i}.  \label{IV.10}
\end{equation}

\begin{corollary}
The complex Randers solution $F$ of generalized ZNP is Douglas if and only if%
\begin{equation}
\overset{h}{G^{i}}=\overset{a}{G^{i}}+\frac{1}{2\tilde{\varepsilon}}\frac{%
\partial \tilde{\varepsilon}}{\partial z^{j}}\eta ^{j}\eta ^{i};\quad
\overset{h}{\theta ^{\ast i}}=\overset{a}{\theta ^{\ast i}}+\frac{1}{\tilde{%
\varepsilon}}\frac{\partial \tilde{\varepsilon}}{\partial \bar{z}^{m}}\left(
h_{0\bar{0}}h^{\bar{m}i}-\bar{\eta}^{m}\eta ^{i}\right) .  \label{IV.11}
\end{equation}
\end{corollary}

\begin{proof} If $F$ is complex Douglas, then we have (\ref{V}) and $%
\Omega _{\bar{m}}=-\frac{\tilde{\varepsilon}}{2f^{2}}\Gamma _{l\bar{r}\bar{m}%
}W^{l}\bar{\eta}^{r},$ $\Omega _{\bar{m}}W^{\bar{m}}=0$ and $\Gamma _{l\bar{r%
}\bar{m}}W^{\bar{m}}=0.$ Substituting these in (\ref{IV.10}) we obtain (%
\ref{IV.11}). Conversely, the relations (\ref{IV.11}) imply the
generalized Berwald property for $F$ and $W_{0}(\frac{\tilde{\varepsilon}}{%
f^{2}}\Gamma _{l\bar{0}\bar{m}}W^{l}+2\Omega _{\bar{m}})]h^{\bar{m}i}=\frac{%
\tilde{\varepsilon}}{f^{2}}\Gamma _{0\bar{0}\bar{m}}\bar{W}^{m}W^{i},$ which
by contraction with $h_{i\bar{0}},$ it gives $\Gamma _{0\bar{0}\bar{m}}\bar{W%
}^{m}=0$ and then, $\Gamma _{l\bar{r}\bar{m}}W^{\bar{m}}=0$, $\Omega _{\bar{m%
}}=-\frac{\tilde{\varepsilon}}{2f^{2}}\Gamma _{l\bar{r}\bar{m}}W^{l}\bar{\eta%
}^{r}$ and $\Omega _{\bar{m}}W^{\bar{m}}=0.$ These give $\theta ^{\ast i}=%
\overset{a}{\theta ^{\ast i}}$ and so, $F$ is Douglas.
\end{proof}

\begin{theorem}
Let $(M,h)$ be a Hermitian manifold. If the complex Randers solution $F$ of
generalized ZNP is generalized Berwald, then $h$ and $a$ are projectively
related if and only $\overset{h}{\theta ^{\ast i}}=\overset{a}{\theta ^{\ast
i}}$. Moreover, any of these assertions implies $\Gamma _{l\bar{s}\bar{m}}%
\bar{W}^{m}=0,$ $\Omega _{\bar{m}}\bar{W}^{m}=\frac{1}{2\tilde{\varepsilon}}%
\frac{\partial \tilde{\varepsilon}}{\partial \bar{z}^{m}}\bar{\eta}^{m}$ and
$\frac{\partial \tilde{\varepsilon}}{\partial \bar{z}^{m}}\bar{W}^{m}=0.$
\end{theorem}

\begin{proof}
According to Corollary 2.1, if $h$ and $a$ are projectively
related then%
\begin{eqnarray}
\dot{\partial}_{\bar{r}}(\overset{a}{\delta _{k}}h)\eta ^{k} &=&\frac{1}{h}(%
\overset{a}{\delta _{k}}h)\eta ^{k}(\dot{\partial}_{\bar{r}}h)\;;\;B^{i}=-%
\frac{1}{h}\overset{a}{\theta ^{\ast k}}(\dot{\partial}_{k}h)\eta ^{i}\;;\;
P =\frac{1}{h}\left[(\overset{a}{\delta _{k}}h)\eta ^{k}+\overset{a}{\theta
^{\ast k}}(\dot{\partial}_{k}h)\right],
\label{IV.12}
\end{eqnarray}%
and the projective change is $\overset{h}{G^{i}}=\overset{a}{G^{i}}+\frac{1%
}{h}(\overset{a}{\delta _{k}}h)\eta ^{k}\eta ^{i}$ , where $h:=\sqrt{h_{0%
\bar{0}}}.$ Under generalized Berwald assumption we have $(\overset{a}{%
\delta _{k}}h)\eta ^{k}=\frac{h}{2\tilde{\varepsilon}}\frac{\partial \tilde{%
\varepsilon}}{\partial z^{k}}\eta ^{k},$ $(\dot{\partial}_{\bar{r}}h)=\frac{%
\tilde{\varepsilon}}{2h}(l_{\bar{r}}-\beta b_{\bar{r}}),$ and then $\dot{%
\partial}_{\bar{r}}(\overset{a}{\delta _{k}}h)\eta ^{k}=\frac{1}{4h}(l_{\bar{%
r}}-\beta b_{\bar{r}})\frac{\partial \tilde{\varepsilon}}{\partial z^{k}}%
\eta ^{k}.$ Thus, the first relation from (\ref{IV.12}) is fulfilled and $%
\overset{h}{\theta ^{\ast i}}=\overset{a}{\theta ^{\ast i}}-\frac{\tilde{%
\varepsilon}}{f^{2}h_{0\bar{0}}}W_{\bar{0}}\Gamma _{0\bar{0}\bar{m}}\bar{W}%
^{m}\eta ^{i},$ with $P=\frac{1}{2\tilde{\varepsilon}}\frac{\partial \tilde{%
\varepsilon}}{\partial z^{k}}\eta ^{k}+\frac{\tilde{\varepsilon}W_{\bar{0}}}{%
2f^{2}h_{0\bar{0}}}\Gamma _{0\bar{0}\bar{m}}\bar{W}^{m}.$ Now, taking into
account (\ref{IV.10}) it results that $T^{i}=0,$ where $T^{i}:=\frac{\tilde{%
\varepsilon}}{f^{2}}\Gamma _{0\bar{0}\bar{m}}\bar{W}^{m}(\frac{1}{h_{0\bar{0}%
}}W_{\bar{0}}\eta ^{i}-W^{i})-\frac{1}{\tilde{\varepsilon}}\frac{\partial
\tilde{\varepsilon}}{\partial \bar{z}^{m}}\bar{\eta}^{m}\eta ^{i}+[\frac{1}{%
\tilde{\varepsilon}}\frac{\partial \tilde{\varepsilon}}{\partial \bar{z}^{m}}%
h_{0\bar{0}}+W_{0}(\frac{\tilde{\varepsilon}}{f^{2}}\Gamma _{l\bar{0}\bar{m}%
}W^{l}+2\Omega _{\bar{m}})]h^{\bar{m}i}.$ Hence, $T^{i}W_{i}=0$ and $[\dot{%
\partial}_{k}(T^{i}W_{i})]W^{k}=0$ and also, $-\frac{1}{h_{0\bar{0}}}W_{\bar{%
0}}(T^{i}W_{i})+[\dot{\partial}_{k}(T^{i}W_{i})]W^{k}=0$ which gives%
\begin{equation}
\left(\frac{\tilde{\varepsilon}}{f^{2}h_{0\bar{0}}}W_{\bar{0}}\Gamma _{0\bar{0}%
\bar{m}}+\Omega _{\bar{m}}\right)\bar{W}^{m}=\frac{1}{2\tilde{\varepsilon}}\frac{%
\partial \tilde{\varepsilon}}{\partial \bar{z}^{m}}\bar{\eta}^{m},
\label{IV.13}
\end{equation}%
because $\alpha ^{2}||b||^{2}\neq |\beta |^{2}$ and $(\dot{\partial}%
_{k}\Omega _{\bar{m}})W^{k}=0.$ Two successive differentiations with respect
to $\eta ^{l}$ and $\bar{\eta}^{r}$ in (\ref{IV.13}), followed by the
contraction with $W^{l}\bar{W}^{r}$ imply $\Gamma _{0\bar{0}\bar{m}}\bar{W}%
^{m}=0,$ and then $\overset{h}{\theta ^{\ast i}}=\overset{a}{\theta ^{\ast i}%
}$ and $\Gamma _{l\bar{s}\bar{m}}\bar{W}^{m}=0.$ Plugging this into (\ref%
{IV.13}), we obtain $\Omega _{\bar{m}}\bar{W}^{m}=\frac{1}{2\tilde{%
\varepsilon}}\frac{\partial \tilde{\varepsilon}}{\partial \bar{z}^{m}}\bar{%
\eta}^{m}$ and then, using again $T^{i}W_{i}=0$ it results that $\frac{\partial
\tilde{\varepsilon}}{\partial \bar{z}^{m}}\bar{W}^{m}=0.$
Conversely, if $\overset{h}{\theta ^{\ast i}}=\overset{a}{\theta ^{\ast i}}$
then, by default as (\ref{IV.10}) $S^{i}=0,$ where
$S^{i}:=-\frac{1}{\tilde{\varepsilon}}\frac{\partial \tilde{\varepsilon}}{%
\partial \bar{z}^{m}}\bar{\eta}^{m}\eta ^{i}-\frac{\tilde{\varepsilon}}{f^{2}%
}\Gamma _{0\bar{0}\bar{m}}\bar{W}^{m}W^{i}+[\frac{1}{\tilde{\varepsilon}}%
\frac{\partial \tilde{\varepsilon}}{\partial \bar{z}^{m}}h_{0\bar{0}}+W_{0}(%
\frac{\tilde{\varepsilon}}{f^{2}}\Gamma _{l\bar{0}\bar{m}}W^{l}+2\Omega _{%
\bar{m}})]h^{\bar{m}i}.$ Hence $S^{i}W_{i}=0$ and $[\dot{\partial}%
_{k}(S^{i}W_{i})]W^{k}=0$ which lead to $\Omega _{\bar{m}}\bar{W}^{m}=\frac{1%
}{2\tilde{\varepsilon}}\frac{\partial \tilde{\varepsilon}}{\partial \bar{z}%
^{m}}\bar{\eta}^{m},$ $\frac{\partial \tilde{\varepsilon}}{\partial \bar{z}%
^{m}}\bar{W}^{m}=0$ and $\Gamma _{k\bar{s}\bar{m}}\bar{W}^{m}=\frac{1}{%
||W||_{h}^{2}}W_{k}\Gamma _{l\bar{0}\bar{m}}W^{l}\bar{W}^{m}.$ Substituting all these terms in $S^{i}=0$, we are thus led to the relation which contracted to  $h_{i\bar{0}}$ gives $\Gamma
_{k\bar{s}\bar{m}}\bar{W}^{m}=0.$ Thus, the conditions (\ref{IV.12}) are
fulfilled and then $h$ and $a$ are projectively related.
\end{proof}

\noindent The above theorem allows us to prove the following

\begin{corollary}
Let $(M,h)$ be a Hermitian manifold. If the complex Randers solution $F$ of
generalized ZNP is generalized Berwald and $h$ and $a$ are projectively
related, then $h$ is K\"{a}hler if and only if $a$ is K\"{a}hler.
\end{corollary}

\begin{corollary}
Let $(M,h)$ be a Hermitian manifold of complex dimension $n=2.$ If the
complex Randers solution $F$ of generalized ZNP is generalized Berwald and $%
h $ and $a$ are projectively related, then $h$ and $a$ are K\"{a}hler.
\end{corollary}

\begin{proof} Under our assumptions, we have $\Gamma _{l\bar{s}\bar{m}}%
\bar{W}^{m}=0$ which can be rewritten as $\Gamma _{l\bar{s}\bar{1}}\bar{W}%
^{1}+\Gamma _{l\bar{s}\bar{2}}\bar{W}^{2}=0,$ with $l,s=1,2.$ Since $\Gamma
_{l\bar{m}\bar{m}}=0$ and $\Gamma _{l\bar{2}\bar{1}}=-\Gamma _{l\bar{1}\bar{2%
}},$ $l,m=1,2$, the last condition is reduced to $\Gamma _{l\bar{1}\bar{2}}\bar{%
W}^{1}=0$ and $\Gamma _{l\bar{1}\bar{2}}\bar{W}^{2}=0$. These give $\Gamma
_{l\bar{1}\bar{2}}$ because at least one of the coefficients $\bar{W}^{m}$
is nonzero. This means that the metric $a$ is K\"{a}hler. Hence, $\overset%
{h}{\theta ^{\ast i}}=\overset{a}{\theta ^{\ast i}}=0,$ i.e. $h$ is also K%
\"{a}hler.
\end{proof}

\begin{theorem}
Let $(M,h)$ be a Hermitian manifold. If the complex Randers solution $F$ to
generalized ZNP is generalized Berwald and $h $ and $a$ are projectively
related, then $F$ is complex Douglas if and only if $\tilde{\varepsilon}%
=const.$
\end{theorem}

\begin{proof} If $F$ is complex Douglas then $\Omega _{\bar{m}}\bar{W}%
^{m}=0$ and by Theorem 4.5, it results $\frac{\partial \tilde{\varepsilon}}{%
\partial \bar{z}^{m}}\bar{\eta}^{m}=0$ and so $\frac{\partial \tilde{%
\varepsilon}}{\partial \bar{z}^{m}}=0,$ i.e. $\tilde{\varepsilon}=const.$
Conversely, if $\tilde{\varepsilon}=const.$ our assumptions and (%
\ref{IV.10}) lead to $\Omega _{\bar{m}}=-\frac{\tilde{\varepsilon}}{2f^{2}}%
\Gamma _{l\bar{r}\bar{m}}W^{l}\bar{\eta}^{r}.$ Also, by Theorem 4.5 we have $%
\Gamma _{l\bar{r}\bar{m}}W^{\bar{m}}=0$ and $\Omega _{\bar{m}}W^{\bar{m}}=0$
which imply $\theta ^{\ast i}=\overset{a}{\theta ^{\ast i}}$, namely $F$ is
Douglas.
\end{proof}

\begin{theorem}
Let $(M,h)$ be a Hermitian manifold of complex dimension $n\geq 2$. If $a$
and the complex Randers solution $F$ of generalized ZNP are projectively
related, then $h$ and $F$ are projectively related if and only if $\tilde{%
\varepsilon}=const.$ Moreover, any of these assertions implies $\overset{h}{%
G^{i}}=G^{i}$, $\overset{h}{\theta ^{\ast i}}=\theta ^{\ast i}$ and%
\begin{equation}
\mathcal{K}_{F}(z,\eta )=\frac{1}{F^{2}\sqrt{W_{0}W_{\bar{0}}}}\left[\tilde{%
\varepsilon}\alpha ^{3}\mathcal{K}_{a}(z,\eta )-\frac{h^{4}}{F}\mathcal{K}%
_{h}(z,\eta )\right].  \label{V.1}
\end{equation}
\end{theorem}

\begin{proof}
We suppose that $h$ and $F$ are projectively related. Since $%
F$ is complex Douglas, Corollary 2.1 and the second formula in (\ref{IV.11}) yield $%
\frac{\partial \tilde{\varepsilon}}{\partial \bar{z}^{m}}\left( h_{0\bar{0}%
}h^{\bar{m}i}-\bar{\eta}^{m}\eta ^{i}\right) =0$. Two successive
differentiations with respect to $\eta ^{k}$ and $\bar{\eta}^{s}$ give $%
\frac{\partial \tilde{\varepsilon}}{\partial \bar{z}^{m}}\left( h_{k\bar{s}%
}h^{\bar{m}i}-\delta _{\bar{s}}^{\bar{m}}\delta _{k}^{i}\right) =0.$ Also, the
contraction with $h^{\bar{s}l}$ leads to $\frac{\partial \tilde{\varepsilon}%
}{\partial \bar{z}^{m}}\left( \delta _{k}^{l}h^{\bar{m}i}-h^{\bar{m}l}\delta
_{k}^{i}\right) =0$ and setting $l=k$ results $(n-1)h^{\bar{m}i}\frac{%
\partial \tilde{\varepsilon}}{\partial \bar{z}^{m}}=0.$ Thus, we have $(n-1)%
\frac{\partial \tilde{\varepsilon}}{\partial \bar{z}^{r}}=0$ what gives that
$\tilde{\varepsilon}=const.,$ so $\overset{h}{G^{i}}=G^{i}$ and $\overset{h}{%
\theta ^{\ast i}}=\theta ^{\ast i}.$ Conversely, if $\tilde{\varepsilon}%
=const.,$ under complex Douglas assumption for $F,$ by (\ref{IV.11}) we get $\overset{h}{G^{i}}=G^{i}$and $\overset{h}{\theta ^{\ast i}}=\theta
^{\ast i}$. Thus, by Corollary 2.1, $h$ and $F$ are projectively related.
Also, since $\tilde{\varepsilon}=const.,$ the formula (\ref{III.0'}) is
reduced to (\ref{V.1}).
\end{proof}
\begin{remark}
\bigskip Under the assumptions of Theorem 4.9, by the formula (\ref{V.1}),
we summarize: i) If $\mathcal{K}_{h}(z,\eta )=0$ then $\mathcal{K}%
_{F}(z,\eta )=\frac{\tilde{\varepsilon}\alpha ^{3}}{F^{2}\sqrt{W_{0}W_{\bar{0%
}}}}\mathcal{K}_{a}(z,\eta );$ ii) If $\mathcal{K}_{h}(z,\eta )>0$, then $%
\mathcal{K}_{F}(z,\eta )<\frac{\tilde{\varepsilon}\alpha ^{3}}{F^{2}\sqrt{%
W_{0}W_{\bar{0}}}}\mathcal{K}_{a}(z,\eta )$ and iii) If $\mathcal{K}%
_{h}(z,\eta )<0$, then $\mathcal{K}_{F}(z,\eta )>\frac{\tilde{\varepsilon}%
\alpha ^{3}}{F^{2}\sqrt{W_{0}W_{\bar{0}}}}\mathcal{K}_{a}(z,\eta ).$
\end{remark}

\noindent Due to Theorem 4.8 and Theorem 4.9 we thus get

\begin{corollary}
Let $(M,h)$ be a Hermitian manifold of complex dimension $n\geq 2$. If $a$
and the complex Randers solution $F$ of generalized ZNP are projectively
related as well as $h$ and $a$ are projectively related, then $h$ and $F$ are
projectively related and $\tilde{\varepsilon}=const.$
\end{corollary}

\begin{corollary}
Let $F$ be the complex Randers solution of generalized ZNP on domain $D$
from $\mathbb{C}^{n}.$ If $F$ is generalized Berwald, then $h$ is K\"{a}hler
and $a$ is locally projectively flat if and only if $a$ is K\"{a}hler and $h$
is locally projectively flat. Moreover, any of these assertions implies that $a$
and $h$ are projectively related, $\mathcal{K}_{h}=c_{1},$ $\mathcal{K}%
_{a}=c_{2},$ $c_{1},c_{2}$ $\in \mathbb{R}$, and%
\begin{equation}
c_{1}=\frac{n\tilde{\varepsilon}+||W||_{h}^{2}}{n\tilde{\varepsilon}^{2}}%
c_{2}-\frac{2}{\tilde{\varepsilon}}h^{\bar{m}j}\left(\frac{\partial \tilde{%
\varepsilon}}{\partial \bar{z}^{m}}\frac{\partial \tilde{\varepsilon}}{%
\partial z^{j}}-\frac{\partial ^{2}\tilde{\varepsilon}}{\partial \bar{z}%
^{m}\partial z^{j}}\right).  \label{IX}
\end{equation}
\end{corollary}

\begin{proof} If $a$ is locally projectively flat, then $\overset{a}{G^{i}}=\frac{1%
}{a}\frac{\partial a}{\partial z^{k}}\eta ^{k}\eta ^{i}$ and $a$ is K\"{a}%
hler. Thus,

\noindent $\frac{1}{h}\frac{\partial h}{\partial z^{k}}\eta ^{k}\eta ^{i}=\frac{1}{2%
\tilde{\varepsilon}}\frac{\partial \tilde{\varepsilon}}{\partial z^{j}}\eta
^{j}\eta ^{i}+\frac{\tilde{\varepsilon}}{2h^{2}}(\frac{\partial \alpha ^{2}}{%
\partial z^{k}}-\frac{\partial |\beta |^{2}}{\partial z^{k}})\eta ^{k}\eta
^{i}$
$=\frac{1}{2\tilde{\varepsilon}}\frac{\partial \tilde{\varepsilon}}{\partial
z^{j}}\eta ^{j}\eta ^{i}+\frac{\tilde{\varepsilon}\alpha ^{2}}{h^{2}}\overset%
{a}{G^{i}}-\frac{\tilde{\varepsilon}}{2h^{2}}[(\overset{a}{\delta _{k}}%
|\beta |^{2})\eta ^{k}+2\bar{\beta}\overset{a}{G^{l}}b_{l}]\eta ^{i}$

\noindent $=\frac{1}{2\tilde{\varepsilon}}\frac{\partial \tilde{\varepsilon}}{\partial
z^{j}}\eta ^{j}\eta ^{i}+\frac{\tilde{\varepsilon}\alpha ^{2}}{h^{2}}\overset%
{a}{G^{i}}-\frac{\tilde{\varepsilon}}{h^{2}}\bar{\beta}\frac{1}{a}\frac{%
\partial a}{\partial z^{k}}\eta ^{k}\eta ^{l}b_{l}\eta ^{i}=\frac{1}{2\tilde{%
\varepsilon}}\frac{\partial \tilde{\varepsilon}}{\partial z^{j}}\eta
^{j}\eta ^{i}+\frac{\tilde{\varepsilon}\alpha ^{2}}{h^{2}}\overset{a}{G^{i}}-%
\frac{\tilde{\varepsilon}|\beta |^{2}}{h^{2}}\overset{a}{G^{i}}$
$=\frac{1}{2\tilde{\varepsilon}}\frac{\partial \tilde{\varepsilon}}{\partial
z^{j}}\eta ^{j}\eta ^{i}+\overset{a}{G^{i}}.$ Then, by relation (\ref{V})
it results that $\overset{h}{G^{i}}=\frac{1}{h}\frac{\partial h}{\partial z^{k}}%
\eta ^{k}\eta ^{i},$ which together with the K\"{a}hler assumption for $h,$
give that $h$ is locally projectively flat.
Conversely, if $h$ is locally projectively then it is K\"{a}hler and $%
\overset{h}{G^{i}}=\frac{1}{h}\frac{\partial h}{\partial z^{k}}\eta ^{k}\eta
^{i}$. Now, owing again to (\ref{V}) we obtain%
\begin{equation}
\overset{a}{G^{i}}=\frac{\tilde{\varepsilon}}{h^{2}}\left(\alpha \frac{\partial a%
}{\partial z^{k}}\eta ^{k}-\bar{\beta}\overset{a}{G^{l}}b_{l}\right)\eta ^{i}.
\label{22}
\end{equation}%
The contraction of (\ref{22}) with $b_{i}$ gives $\overset{a}{G^{i}}b_{i}=%
\frac{\beta }{\alpha }\frac{\partial a}{\partial z^{k}}\eta ^{k},$ which
substituted in (\ref{22}) yields $\overset{a}{G^{i}}=\frac{1}{a}\frac{%
\partial a}{\partial z^{k}}\eta ^{k}\eta ^{i}$, and with assumption that $a$
is K\"{a}hler, we thus get the local projective flatness for $a$.
Therefore, $\overset{h}{\theta ^{\ast i}}=\overset{a}{\theta ^{\ast i}}=0$
and according to Theorem 4.9, $a$ and $h$ are projectively related. Also,
according to Corrolary 4.3, we have $\overset{h}{W_{j\bar{k}h}^{i}}=\overset{%
a}{W_{j\bar{k}h}^{i}}=0,$ and so, by Theorem 4.2, it results that $\mathcal{K}%
_{h}=c_{1},$ $\mathcal{K}_{a}=c_{2}$ with $c_{1},c_{2}$ $\in \mathbb{R},$
and
\begin{equation*}
c_{1}h_{j\bar{m}}=c_{2}a_{j\bar{m}}-\frac{2}{\tilde{\varepsilon}}\left(\frac{%
\partial \tilde{\varepsilon}}{\partial \bar{z}^{m}}\frac{\partial \tilde{%
\varepsilon}}{\partial z^{j}}-\frac{\partial ^{2}\tilde{\varepsilon}}{%
\partial \bar{z}^{m}\partial z^{j}}\right)
\end{equation*}%
which multiplied by $h^{\bar{m}j}$ leads to (\ref{IX}).
\end{proof}

Owing to Theorem 4.9, Proposition 4.2 and Corollary 4.12 we have the
following result

\begin{theorem}
Let $F$ be a complex Randers solution of generalized ZNP on domain $D$
from $\mathbb{C}^{n}.$ Then, $h$ is K\"{a}hler and $F$ is locally
projectively flat if and only if $h$ is locally projectively flat and $F$ is
complex Berwald. Moreover, any of these assertions implies $\tilde{%
\varepsilon}=const.,$ $h,$ $a$ and $F$ are projectively related and $%
\mathcal{K}_{F}=\mathcal{K}_{a}=\mathcal{K}_{h}=0.$
\end{theorem}

Next, we exemplify such a case through a model given by the generalized
navigation data $(h$, $W$, $||u||_{h}=const.\leq 1)$, with $h$ the standard
Euclidean metric, ($h_{i\bar{j}}=\delta _{i\bar{j}})$, the mild wind $W$ with
constant components, i.e. $W^{k}=\lambda _{k},$ $\lambda _{k}\in \mathbb{C}$%
, $k=\overline{1,n},$ and $\cos \varphi =-1.$ Thus, these lead to $%
||W||_{h}^{2}=\sum_{k=1}|\lambda _{k}|^{2}$, $\tilde{\varepsilon}=const.,$ $%
b_{i}=\frac{1}{\tilde{\varepsilon}}\bar{\lambda}_{i},$ $a_{i\bar{j}}=\frac{%
\delta _{i\bar{j}}}{\tilde{\varepsilon}}+\frac{1}{\tilde{\varepsilon}^{2}}%
\bar{\lambda}_{i}\lambda _{k}$ \ and so, the fundamental metric tensor $g_{i%
\bar{j}}$ of the complex Randers solution $F$ depends only $\eta $, i.e. $F$
is locally Minkowski. Moreover, it is locally projectively flat and $%
\mathcal{K}_{F}=\mathcal{K}_{a}=\mathcal{K}_{h}=0.$

\subsection{Geodesics of conformal solutions to Zermelo's problem}

According to Theorem 3.11 the solutions of generalized ZNP are conformal to
the background metric $h,$ if either $W=0$ or $W\neq 0$ and $\cos \varphi
=0. $ Our next goal is to find the projective relationship between $h$ and $%
F=\rho (z)h$ ($h:=\sqrt{h_{0\bar{0}}}$), where $\rho (z)=\frac{1}{%
||u(z)||_{h}}$ when $W=0$ or $\rho (z)=\frac{1}{\sqrt{\tilde{\varepsilon}}},$
$\tilde{\varepsilon}:=||u(z)||_{h}^{2}-||W||_{h}^{2}$ when $W\neq 0$ and $%
\cos \varphi =0.$ After some computation we find the following connections
between $h$ and $F=\rho (z)h$%
\begin{eqnarray}
\overset{h}{G^{i}} &=&G^{i}-\frac{1}{2\rho ^{2}}\frac{\partial \rho ^{2}}{%
\partial z^{j}}\eta ^{j}\eta ^{i};\;\;\overset{h}{\theta ^{\ast i}}=\theta
^{\ast i}-\frac{1}{\rho ^{2}}\frac{\partial \rho ^{2}}{\partial \bar{z}^{m}}%
\left( h_{0\bar{0}}h^{\bar{m}i}-\bar{\eta}^{m}\eta ^{i}\right) ;
\label{IV.15'} \\
\mathcal{K}_{F}(z,\eta ) &=&\frac{1}{\rho ^{2}}\mathcal{K}_{h}(z,\eta )-%
\frac{2}{F^{2}}\left(\frac{1}{\rho ^{2}}\frac{\partial \rho ^{2}}{\partial \bar{z}%
^{m}}\frac{\partial \rho ^{2}}{\partial z^{j}}-\frac{\partial ^{2}\rho ^{2}}{%
\partial \bar{z}^{m}\partial z^{j}}\right)\eta ^{j}\bar{\eta}^{m},  \notag
\end{eqnarray}%
where the functions $G^{i}$ and $\theta ^{\ast i}$ correspond to the
conformal solution $F.$

\begin{theorem}
Let $(M,h)$ be a Hermitian manifold of complex dimension $n\geq 2$. Then,
the conformal solution $F$ of generalized ZNP and $h$ are projectively
related if and only if $\rho =const.$ (i.e. $||u(z)||_{h}=const.$ when $W=0$
or $\tilde{\varepsilon}:=||u(z)||_{h}^{2}-||W||_{h}^{2}=const.$ when $W\neq
0 $ and $\cos \varphi =0).$ Moreover, any of these assertions implies $%
\overset{h}{G^{i}}=G^{i}$, $\overset{h}{\theta ^{\ast i}}=\theta ^{\ast i}$
and $\mathcal{K}_{F}=\frac{1}{\rho ^{2}}\mathcal{K}_{h}.$
\end{theorem}

\begin{proof} We suppose that $h$ and $F$ are projectively related. Due to
Corollary 2.1 and the second formula in (\ref{IV.15'}) we have $\frac{\partial \rho
^{2}}{\partial \bar{z}^{m}}\left( h_{0\bar{0}}h^{\bar{m}i}-\bar{\eta}%
^{m}\eta ^{i}\right) =0$ which by differentiations with respect to $\eta
^{k} $ and $\bar{\eta}^{s}$ and contraction with $h^{\bar{s}l}$ leads to $%
\frac{\partial \rho ^{2}}{\partial \bar{z}^{m}}\left( \delta _{k}^{l}h^{\bar{%
m}i}-h^{\bar{m}l}\delta _{k}^{i}\right) =0$. Setting $l=k$, it results that $%
(n-1)h^{\bar{m}i}\frac{\partial \rho ^{2}}{\partial \bar{z}^{m}}=0.$ Thus,
we have $(n-1)\frac{\partial \rho ^{2}}{\partial \bar{z}^{r}}=0$ what gives
that $\rho ^{2}=const.,$ so $\overset{h}{G^{i}}=G^{i}$ and $\overset{h}{%
\theta ^{\ast i}}=\theta ^{\ast i}.$ Conversely, if $\rho ^{2}=const.,$ by (%
\ref{IV.15'}) it results that $\overset{h}{G^{i}}=G^{i},$ $\overset{h}{\theta
^{\ast i}}=\theta ^{\ast i}$ and $\mathcal{K}_{F}=\frac{1}{\rho ^{2}}%
\mathcal{K}_{h}$ and applying again Corollary 2.1 we obtain that $h$ and $F$ are
projectively related.
\end{proof}

\begin{corollary}
Let $(M,h)$ be a Hermitian manifold of complex dimension $n\geq 2$. Then, the
conformal solution $F$ of generalized ZNP and $h$ are projectively related
if and only if $F$ and $h$ are homothetic.
\end{corollary}

\noindent The generalized navigation data $(h$, $W$, $||u||_{h}=const.\leq
1) $ with $\cos \varphi =0$, $W^{k}=\lambda _{k},$ $\lambda _{k}\in \mathbb{C%
},$ and $||W||_{h}=const.$ lead to the homothetic solutions to $h$. {Indeed,
our assumptions give }$\tilde{\varepsilon}=const.$ and due to the above
Corollary, the conformal solution $F$ is homothetic to $h$.

\begin{corollary}
Let $(M,h)$ be a K\"{a}hler manifold of complex dimension $n\geq 2$. Then,
the conformal solution $F$ of generalized ZNP is K\"{a}hler if and only if $%
\rho =const.$
\end{corollary}

\begin{proof} If $F$ is K\"{a}hler, by the second formula (\ref{IV.15'}), it
results that $\frac{\partial \rho ^{2}}{\partial \bar{z}^{m}}\left( h_{0\bar{0}%
}h^{\bar{m}i}-\bar{\eta}^{m}\eta ^{i}\right) =0$ which leads to $\rho
=const. $ Applying again the second formula (\ref{IV.15'}) the converse is also obtained.
\end{proof}

\begin{corollary}
Let $F$ be a conformal solution to generalized ZNP on domain $D$
from $\mathbb{C}^{n}$, $n\geq 2$. If $F$ is homothetic to $h,$ then $h$ is
projectively flat if and only if $F$ is projectively flat. Moreover, any of
these assertions implies $\mathcal{K}_{F}=\frac{1}{\rho ^{2}}c,$ where $%
\mathcal{K}_{h}=c,$ $c\in \mathbb{R}.$
\end{corollary}

\begin{proof} If $h$ is locally projectively flat, then $\overset{h}{G^{i}}=\frac{1%
}{h}\frac{\partial h}{\partial z^{k}}\eta ^{k}\eta ^{i}$ and $h$ is K\"{a}%
hler. Thus, $\frac{1}{F}\frac{\partial F}{\partial z^{k}}\eta ^{k}\eta ^{i}=%
\frac{1}{\rho }\frac{\partial \rho }{\partial z^{j}}\eta ^{j}\eta ^{i}+\frac{%
1}{h}\frac{\partial h}{\partial z^{j}}\eta ^{j}\eta ^{i}=\overset{h}{G^{i}},$
since $\rho =const.$ and, by the first two formulae of (\ref{IV.15'}) it
results that $G^{i}=\frac{1}{F}\frac{\partial F}{\partial z^{k}}\eta ^{k}\eta
^{i} $ and $F$ is K\"{a}hler, i.e. $F$ is locally projectively flat.

Conversely, if $F$ is locally projectively flat then it is K\"{a}hler and $G^{i}=%
\frac{1}{F}\frac{\partial F}{\partial z^{k}}\eta ^{k}\eta ^{i}$. Now, using
again the first two formulae of (\ref{IV.15'}) and $\rho =const.$, we obtain
the local projective flatness for $h.$ Moreover, due to Theorem 2.2
we obtain $\mathcal{K}_{h}=c,$ $c\in \mathbb{R},$ and so, $\mathcal{K}_{F}=%
\frac{1}{\rho ^{2}}c.$
\end{proof}

Now, on a complex manifold $M$ we consider the following generalized
navigation data: $(h$, $W$, $||u(z)||_{h})$ with $\cos \varphi =-1$ and $(h$%
, $\tilde{W}$, $||\tilde{u}(z)||_{h})$ with $\cos \varphi =0$ and $\tilde{W}%
\neq 0$, which produce a complex Randers solution $F$ and a conformal
solution $\tilde{F},$ respectively. A connection between the solutions $F$
and $\tilde{F}$ can be formulated as follows

\begin{proposition}
Let $(M,h)$ be a Hermitian manifold of complex dimension $n\geq 2,$ and $F$
and $\tilde{F}$ the solutions obtained by generalized navigation data $(h$, $%
W$, $||u(z)||_{h})$ with $\cos \varphi =-1$ and $(h$, $\tilde{W}$, $||\tilde{%
u}(z)||_{h})$ with $\cos \varphi =0$ and $\tilde{W}\neq 0$. If $\ F$ is
complex Douglas and $h$ is projectively related with $F$ and homothetic with
$\tilde{F}$, then $F$ and $\tilde{F}$ are projectively related.
\end{proposition}

\begin{proof} First, owing to Theorem 4.9, we have $\varepsilon
_{1}:=||u(z)||_{h}^{2}-||W||_{h}^{2}=const.$ and $\overset{h}{G^{i}}=G^{i}$,
$\overset{h}{\theta ^{\ast i}}=\theta ^{\ast i}.$ Second, acording to
Theorem 4.14 and Corollary 4.15, $\varepsilon _{2}:=||\tilde{u}%
(z)||_{h}^{2}-||\tilde{W}||_{h}^{2}=const.$ and $\overset{h}{G^{i}}=\tilde{G}%
^{i}$, $\overset{h}{\theta ^{\ast i}}=\tilde{\theta}^{\ast i}$. Hence, $G^{i}=%
\tilde{G}^{i}$, $\theta ^{\ast i}=\tilde{\theta}^{\ast i}$ and so, $F$ and $%
\tilde{F}$ are projectively related.
\end{proof}

\noindent To conclude, let us also mention that the adequate investigation on conformal and weakly conformal real Finsler geometry can be found in particular in \cite{rafie, matveev}.

\subsection{Some examples}

In what follows we consider a mild wind $W$ with non constant components
and a Hermitian manifold $M$ is represented by Hartogs triangle $D=\left\{ (z,w)\in
\mathbb{C}^{2},\;|w|<|z|<1\right\}$. As being in dimension two we denote the local position coordinates $(z^1, z^2)$ by $(z, w)$.

\textbf{I.} \textbf{Example for the case when }$a$\textbf{\ and }$F$\textbf{\
are projectively related} (with\textbf{\ }$||u||_{h}^{2}=\frac{|z|^{2}}{2}$
and $\cos \varphi =-1$). \bigskip \noindent On the Hartogs triangle
we consider the generalized navigation data $\left(h,||u||_{h}^{2}=\frac{|z|^{2}}{%
2},W=-\frac{|z|^{2}-|w|^{2}}{2z}\frac{\partial }{\partial w}\right)$ and $\cos
\varphi =-1,$ with $h_{i\overline{j}}=|z|^{2}\frac{\partial ^{2}}{\partial
z^{i}\partial \overline{z}^{j}}\left( \log \frac{1}{\left( 1-|z|^{2}\right)
\left( |z|^{2}-|w|^{2}\right) }\right) $. These imply $||W||_{h}^{2}=\frac{%
|z|^{2}}{4}$, $W_{1}=\frac{w|z|^{2}}{2(|z|^{2}-|w|^{2})}$, $W_{2}=\frac{%
-z|z|^{2}}{2(|z|^{2}-|w|^{2})},$ $\tilde{\varepsilon}=\frac{|z|^{2}}{4}$ and
then, $b_{1}=\frac{2w}{|z|^{2}-|w|^{2}},$ $b_{2}=\frac{-2z}{|z|^{2}-|w|^{2}}$
as well as %
\begin{equation}
\left( a_{j\bar{k}}(z)\right) _{j,k=1,2}=8\left(
\begin{array}{ll}
\frac{1}{2\left( 1-|z|^{2}\right) ^{2}}+\frac{|w|^{2}}{(|z|^{2}-|w|^{2})^{2}}
& \frac{-w\bar{z}}{(|z|^{2}-|w|^{2})^{2}} \\
\;\;\;\;\;\;\;\;\;\frac{-z\bar{w}}{(|z|^{2}-|w|^{2})^{2}} & \frac{|z|^{2}}{%
(|z|^{2}-|w|^{2})^{2}}%
\end{array}%
\right) .  \label{A}
\end{equation}%
Consequently, we obtain $b^{1}=0,$ $b^{2}=-\frac{|z|^{2}-|w|^{2}}{4z}$ and$%
\;||b||^{2}=\frac{1}{2}.$ Thus, the solution of the Zermelo navigation
problem on $D$ is the complex Randers metric $F=\alpha +|\beta |,$ with $%
\alpha ^{2}=a_{i\overline{j}}\eta ^{i}\overline{\eta }^{j}$ and $|\beta
|^{2}=\frac{4|w\eta ^{1}-z\eta ^{2}|^{2}}{(|z|^{2}-|w|^{2})^{2}}.$ Also, $%
A=0$ (see page 10), i.e. $F$ is generalized Berwald and $\overset{a}{G^{i}}=G^{i}.$
Corresponding to $a$, we have
\begin{equation}
\overset{a}{G^{1}}=\frac{\overline{z}(\eta ^{1})^{2}}{1-|z|^{2}};\;\overset{a%
}{G^{2}}=\frac{\overline{z}w(1-|w|^{2})(\eta ^{1})^{2}}{z(1-|z|^{2})\left(
|z|^{2}-|w|^{2}\right) }-\frac{(|z|^{2}+|w|^{2})\eta ^{1}\eta ^{2}}{z\left(
|z|^{2}-|w|^{2}\right) }+\frac{\overline{w}(\eta ^{2})^{2}}{|z|^{2}-|w|^{2}}.
\label{B}
\end{equation}%
Moreover, we obtain that $a{\ }$is K\"{a}hler ($\Gamma _{l\bar{r}\bar{m}}=0)$
and $\Omega _{\bar{m}}b^{\bar{m}}=0$ which give us that $F$ is a complex
Berwald metric. Thus, $\overset{a}{\theta ^{\ast i}}=\theta ^{\ast i}=0$ and
$a$ and $F$ are projectively related and the corresponding geodesics are
solutions to the differential system%
\begin{equation}
\left\{
\begin{array}{l}
\ddot{\gamma}^{1}+\frac{2\bar{\gamma}^{1}(\dot{\gamma}^{1})^{2}}{1-|\gamma
^{1}|^{2}}=0 \\
\ddot{\gamma}^{2}+\frac{2\bar{\gamma}^{1}\gamma ^{2}\left( 1-|\gamma
^{2}|^{2}\right) (\dot{\gamma}^{1})^{2}}{\gamma ^{1}(1-|\gamma
^{1}|^{2})\left( |\gamma ^{1}|^{2}-|\gamma ^{2}|^{2}\right) }-\frac{%
2(|\gamma ^{1}|^{2}+|\gamma ^{2}|^{2})\dot{\gamma}^{1}\dot{\gamma}^{2}}{%
\gamma ^{1}\left( |\gamma ^{1}|^{2}-|\gamma ^{2}|^{2}\right) }+\frac{2\bar{%
\gamma}^{2}(\dot{\gamma}^{2})^{2}}{|\gamma ^{1}|^{2}-|\gamma ^{2}|^{2}}=0%
\end{array}%
\right. ,  \label{C}
\end{equation}%
By searching the solutions with the properties $\gamma ^{1}=\lambda ,$ $\lambda
\in \mathbb{R},$ and $\gamma ^{2}=\bar{\gamma}^{2},$ the system (\ref{C})
is reduced to the equation $\ddot{\gamma}^{2}+\frac{2\gamma ^{2}(\dot{\gamma}%
^{2})^{2}}{\lambda ^{2}-(\gamma ^{2})^{2}}=0.$ It results that $\dot{\gamma}%
^{2}=k[\lambda ^{2}-(\gamma ^{2})^{2}],$ $k\in \mathbb{R}$, and then $\gamma
^{2}=\frac{\lambda (\mu e^{2\lambda kt}-1)}{\mu e^{2\lambda kt}+1},$ $\mu
\in \mathbb{R}$. So, $\dot{\gamma}=v=k[\lambda ^{2}-(\gamma ^{2})^{2}]\frac{%
\partial }{\partial w},$ $W=-\frac{\lambda ^{2}-(\gamma ^{2})^{2}}{2\lambda }%
\frac{\partial }{\partial w}$ and $u=v-W=[\lambda ^{2}-(\gamma ^{2})^{2}](k+%
\frac{1}{2\lambda })\frac{\partial }{\partial w}.$ The initial conditions $%
||u||_{h}^{2}=\frac{\lambda ^{2}}{2}$ and $\cos \varphi =-1$ lead to $%
\lambda k=\frac{\sqrt{2}-1}{2}.$ Thus, a solution to the differential system
(\ref{C}) is $\gamma (t)=\left(\lambda ,\frac{\lambda (\mu e^{(\sqrt{2}-1)t}-1)}{%
\mu e^{(\sqrt{2}-1)t}+1}\right).$ Here, $W,$ $v$ and $u$ are collinear, their
components depend on $t$ (they are not constant), however $||v||_{h}=\frac{\sqrt{2}%
-1}{2}\gamma ^{1}(t)$ is constant because $\gamma ^{1}=\lambda .$ Moreover, $%
l${$_{F}(\gamma )=1>l_{a}(\gamma )=2-\sqrt{2}$ and $l_{h}(\gamma )=$}$\frac{%
\sqrt{2}-1}{2}\lambda.$ The metrics $h$ and $F$ are not projectively related because $%
\tilde{\varepsilon}$ is not constant.

Next, in order to complement our findings let us consider the same data $\left(h,||u||_{h}^{2}=\frac{|z|^{2}}{2}\right)$ in the absence of the wind. Therefore, it results that $\tilde{\varepsilon}%
=||u||_{h}^{2}=\frac{|z|^{2}}{2}$. So, the solution of the Zermelo
navigation problem on Hartogs triangle is the conformal metric $\tilde{F}=%
\frac{\sqrt{2}}{|z|}h,$ to $\tilde{g}_{j\bar{k}}=\frac{2}{|z|^{2}}h_{j\bar{%
k}}$ which is not homothetic to $h$ due to the fact that $\tilde{\varepsilon}$ is not
constant. As above, we find that the geodesics corresponding to $\tilde{F}$
are $\gamma _{1,2}(t)=\left(\lambda ,\frac{\lambda (\mu e^{\pm \sqrt{2}t}-1)}{\mu
e^{\pm \sqrt{2}t}+1}\right),$ $k\lambda =\pm \frac{\sqrt{2}}{2},$ $\mu \in \mathbb{%
R}$ and $v=u=k[\lambda ^{2}-(\gamma ^{2})^{2}]\frac{\partial }{\partial w}.$
It results that $l_{\tilde{F}}(\gamma _{1,2})=1$ and $l_{h}(\gamma
_{1,2})=|\lambda |\frac{\sqrt{2}}{2}$\ on $[0,1],$ but
here $\gamma _{1,2}$\ are not the geodesics of $h.$ Note that the solutions $F$ and $\tilde{F}$ are projectively related.

\

\textbf{II.} \textbf{Example for the case when} $h$\textbf{\ and }$%
F$\textbf{\ are conformal } (with $||u||_{h}^{2}=\frac{|z|^{2}}{2}$
and $\cos \varphi =0$).

Again, on the Hartogs triangle we consider the Hermitian metric $h_{i%
\overline{j}}=|z|^{2}\frac{\partial ^{2}}{\partial z^{i}\partial \overline{z}%
^{j}}\left( \log \frac{1}{\left( 1-|z|^{2}\right) \left(
|z|^{2}-|w|^{2}\right) }\right) $ and the generalized navigation data $%
\left(h,|||u||_{h}^{2}=\frac{|z|^{2}}{2},W=i\frac{|z|^{2}-|w|^{2}}{2z}\frac{%
\partial }{\partial w}\right)$ and $\cos \varphi =0.$ Hence $||W||_{h}^{2}=\frac{%
|z|^{2}}{4}$, $\tilde{\varepsilon}=\frac{|z|^{2}}{4}$ and $g_{j\bar{k}}=%
\frac{4}{|z|^{2}}h_{j\bar{k}}$ and the solution of the Zermelo navigation
problem on $D$ is the Hermitian metric $F=\frac{2}{|z|}h$ which is conformal
to $h.$ Also, we obtain $G^{i}$ given by (\ref{B}) and $F{\ }$is K\"{a}%
hler. The corresponding geodesics are the solutions of the differential system (%
\ref{C}) which admits $\gamma (t)=\left(\lambda ,\frac{\lambda (\mu
e^{2\lambda kt}-1)}{\mu e^{2\lambda kt}+1}\right),$ $\lambda ,\mu ,k\in \mathbb{R}$%
, and $\dot{\gamma}=v=k[\lambda ^{2}-(\gamma ^{2})^{2}]\frac{\partial }{%
\partial w},$ $W=i\frac{\lambda ^{2}-(\gamma ^{2})^{2}}{2\lambda }\frac{%
\partial }{\partial w}$, with $u=v-W=[\lambda ^{2}-(\gamma ^{2})^{2}](k-\frac{i%
}{2\lambda })\frac{\partial }{\partial w}.$ The initial conditions $%
||u||_{h}^{2}=\frac{\lambda ^{2}}{2}$ and $\cos \varphi =0$ lead to $\lambda
k=\pm \frac{1}{2}.$ Thus, the geodesics corresponding to the conformal
metric $F$ are $\gamma (t)=\left(\lambda ,\frac{\lambda (\mu e^{\pm t}-1)}{\mu
e^{\pm t}+1}\right)$ and $W,$ $v$ and $u$ are not collinear $(W$ and $v$ are
orthogonal) and of not constant components. We also obtain that {$%
l_{F}(\gamma )=1$ and $l_{h}(\gamma )=$}$\frac{|\lambda |}{2}.$ Moreover, the angle between $u$ and $W$ is constant. Indeed, let $\phi :=\measuredangle (u,W)$ and $\text{Re} h(u,\bar{W})=||u||_{h}||W||_{h}\cos
\phi .$ Since $h(u,\bar{W})=\frac{\lambda ^{2}}{4}(-1\pm i)$ it results that Re$%
h(u,\bar{W})=-\frac{\lambda ^{2}}{4}.$ So, $\cos \phi =-\frac{\sqrt{2}}{2},$
i.e. $\measuredangle (u,W)=\frac{3\pi }{4}=const.$

The same navigation data in the absence of the wind $W$
lead to the same solution as in example I. Here we have $l_{\tilde{F%
}}(\gamma _{1,2})=1$ and $l_{h}(\gamma _{1,2})=|\lambda
|\frac{\sqrt{2}}{2}$ on $[0,1].$ Also, $\gamma _{1,2}$
are not the geodesics of $h.$

\

\textbf{III. Example for the case }$h$\textbf{, }$a$\textbf{\ and }$F$%
\textbf{\ are projectively related } (with\textbf{\ }$||u||_{h}^{2}=\frac{1}{2%
}$ and $\cos \varphi =-1$). \noindent On the Hartogs
triangle we consider the generalized navigation data $(h,||u||_{h}^{2}=\frac{%
1}{2},W=-\frac{|z|^{2}-|w|^{2}}{2z}\frac{\partial }{\partial w})$ and $\cos
\varphi =-1,$ with $h_{i\overline{j}}=\frac{\partial ^{2}}{\partial
z^{i}\partial \overline{z}^{j}}\left( \log \frac{1}{\left( 1-|z|^{2}\right)
\left( |z|^{2}-|w|^{2}\right) }\right) $. These imply $||W||_{h}^{2}=\frac{1%
}{4}$, $W_{1}=\frac{w}{2(|z|^{2}-|w|^{2})}$, $W_{2}=\frac{-z}{%
2(|z|^{2}-|w|^{2})},$ $\tilde{\varepsilon}=\frac{1}{4}$ and then, $b_{1}=%
\frac{2w}{|z|^{2}-|w|^{2}},$ $b_{2}=\frac{-2z}{|z|^{2}-|w|^{2}}$ and $a_{j%
\bar{k}}(z)$ has the same form as in (\ref{A}). Also, $b^{1}=0,$ $b^{2}=-%
\frac{|z|^{2}-|w|^{2}}{4z}$ and$\;||b||^{2}=\frac{1}{2}.$ Thus, the solution
of the Zermelo navigation problem on $D$ is the same complex Berwald metric
as in example I. The difference is that $\overset{h}{G^{i}}=\overset{a}{%
G^{i}}=G^{i}$ with the same form as in (\ref{B}). Another change is that
here $h$ is K\"{a}hler. Thus, $\overset{h}{\theta ^{\ast i}}=\overset{a}{%
\theta ^{\ast i}}=\theta ^{\ast i}=0$ and $h$, $a$ and $F$ are projectively
related and the corresponding geodesics are solutions of the differential
system (\ref{C}) which admits the solution $\gamma (t)=\left(\lambda ,\frac{%
\lambda (\mu e^{2\lambda kt}-1)}{\mu e^{2\lambda kt}+1}\right),$ $\lambda ,\mu \in
\mathbb{R}$, with $\dot{\gamma}=v=k[\lambda ^{2}-(\gamma ^{2})^{2}]\frac{%
\partial }{\partial w},$ $W=-\frac{\lambda ^{2}-(\gamma ^{2})^{2}}{2\lambda }%
\frac{\partial }{\partial w}$ and $u=v-W=[\lambda ^{2}-(\gamma ^{2})^{2}](k+%
\frac{1}{2\lambda })\frac{\partial }{\partial w}.$ The initial conditions $%
||u||_{h}^{2}=\frac{1}{2}$ and $\cos \varphi =-1$ lead to $\lambda k=\frac{%
\sqrt{2}-1}{2}.$ Thus, a solution to the differential system (\ref{C}) is $%
\gamma (t)=\left(\lambda ,\frac{\lambda (\mu e^{(\sqrt{2}-1)t}-1)}{\mu e^{(\sqrt{2%
}-1)t}+1}\right).$ Here, the vectors $W,$ $v$ and $u$ are collinear. Their components depend on $t$ (they are not constant), however $||v||_{h}$ is constant. Moreover, {$%
l_{F}(\gamma )=1>l_{a}(\gamma )=2-\sqrt{2}>l_{h}(\gamma )=$}$\frac{\sqrt{2}-1%
}{2}.$

Next, let us consider the same data $(h,||u||_{h}^{2}=\frac{1}{2})$ in the
absence of the wind. Therefore, it results that $\tilde{\varepsilon}%
=||u||_{h}^{2}=\frac{1}{2}$ and $\tilde{g}_{j\bar{k}}(z)=2h_{j\bar{k}}.$ Thus,
the solution of the Zermelo navigation problem on Hartogs triangle is the
conformal metric $\tilde{F}=\sqrt{2}h,$ homothetic to $h.$ As above, the
geodesics corresponding to $\tilde{F}$ are $\gamma _{1,2}(t)=\left(\lambda ,\frac{%
\lambda (\mu e^{\pm \sqrt{2}t}-1)}{\mu e^{\pm \sqrt{2}t}+1}\right),$ $\lambda
k=\pm \frac{\sqrt{2}}{2},$ $\mu \in \mathbb{R}$ and $v=u=k[\lambda
^{2}-(\gamma ^{2})^{2}]\frac{\partial }{\partial w}.$ It results that $l_{\tilde{%
F}}(\gamma _{1,2})=1$ and {$l_{h}(\gamma _{1,2})=\frac{\sqrt{2}}{2%
}$} on $[0,1].$ So, $l_{F}(\gamma )=1>l_{h}(\gamma
_{1,2})=\frac{\sqrt{2}}{2}.$

\

\textbf{IV. Example for the case when }$h$\textbf{\ and }$F$%
\textbf{\ are homothetic. }Now, on the Hartogs triangle with the Hermitian
metric $h_{i\overline{j}}=\frac{\partial ^{2}}{\partial z^{i}\partial
\overline{z}^{j}}\left( \log \frac{1}{\left( 1-|z|^{2}\right) \left(
|z|^{2}-|w|^{2}\right) }\right) $ we consider the generalized navigation
data $(h,||u||_{h}^{2}=\frac{1}{2},W=i\frac{|z|^{2}-|w|^{2}}{2z}\frac{%
\partial }{\partial w})$ and $\cos \varphi =0.$ Hence $||W||_{h}^{2}=\frac{1%
}{4}$, $\tilde{\varepsilon}=\frac{1}{4}$ and $g_{j\bar{k}}=a_{j\bar{k}}=4h_{j%
\bar{k}}$ and the solution of the Zermelo navigation problem on $D$ is the
Hermitian metric $F=2h=2\sqrt{h_{i\overline{j}}\eta ^{i}\overline{\eta }^{j}}
$ which is homothetic to $h.$ Also, we obtain that $\overset{h}{G^{i}}=G^{i}$
given by (\ref{B}); $h$ and $F{\ }$are K\"{a}hler. The corresponding
geodesics are the solutions of the differential system (\ref{C}) which admits
 $\gamma (t)=\left(\lambda ,\frac{\lambda (\mu e^{2\lambda kt}-1)}{%
\mu e^{2\lambda kt}+1}\right),$ $\lambda ,\mu ,k\in \mathbb{R}$, and $\dot{\gamma}%
=v=k\left[\lambda ^{2}-(\gamma ^{2})^{2}\right]\frac{\partial }{\partial w},$ $W=i\frac{%
\lambda ^{2}-(\gamma ^{2})^{2}}{2\lambda }\frac{\partial }{\partial w}$, with $%
u=v-W=[\lambda ^{2}-(\gamma ^{2})^{2}](k-\frac{i}{2\lambda })\frac{\partial
}{\partial w}.$ The initial conditions $||u||_{h}^{2}=\frac{1}{2}$ and $\cos
\varphi =0$ lead to $\lambda k=\pm \frac{1}{2}.$ Thus, the geodesics
corresponding to the homothetic metrics $h$ and $F$ are $\gamma (t)=\left(\lambda
,\frac{\lambda (\mu e^{\pm t}-1)}{\mu e^{\pm t}+1}\right)$ and the vectors $W,$ $v$ and $u$
are not collinear $(W$ and $v$ are orthogonal) and of not
constant components. Also, {$l_{F}(\gamma )=1>l_{h}(\gamma )=$}$\frac{1}{2}%
. $

Finally, we consider the same Zermelo's structure $(h,||u||_{h}^{2}=\frac{1}{2})$ in the absence of the wind $W$. This leads
to the same solution $\tilde{F}$ as in Example III. Here we get $l_{\tilde{F}%
}(\gamma _{1,2})=1$ and {$l_{h}(\gamma _{1,2})=\frac{%
\sqrt{2}}{2}$} on $[0,1].$ Therefore, $l_{F}(\gamma
)=1>l_{h}(\gamma _{1,2})=\frac{\sqrt{2}}{2}.$

\

\noindent We conclude by remarking that
\begin{remark}
Owing to Proposition 4.18 we have that the solutions obtained in Examples III and IV, with the considered perturbations, are projectively related.
\end{remark}

\

\end{document}